\documentclass[10pt]{amsart}

\RequirePackage{mathrsfs} \let\mathcal\mathscr
\usepackage{enumerate}
\usepackage[utf8]{inputenc}
\usepackage{color}
\usepackage{amssymb,amsmath,amsfonts,amsthm}
\usepackage{hyperref}

\usepackage{graphicx}
\usepackage{amsfonts,amsmath,latexsym,amssymb,amsthm}
\usepackage{verbatim}
\usepackage{mathrsfs}

\newcommand{\Oseen}{\mathcal{O}}

\newcommand{\IZ}{\mathbb{Z}}
\newcommand{\IQ}{\mathbb{Q}}

\newcommand{\ord}{\text{\rm ord}}

\newcommand{\B}{\mathscr{B}}
\newcommand{\A}{\mathscr{A}}
\newcommand{\DK}{D_K}
\newcommand{\del}{\eta(K/F)}
\newcommand{\ka}{\kappa}

\newcommand{\ro}{\rho}
\newcommand{\E}{\mathcal{E}}
\newcommand{\Pro}{\mathbb{P}}
\newcommand{\Var}{\text{Var}}
\newcommand{\p}{\mathfrak{p}}
\newcommand{\q}{\mathfrak{q}}
\newcommand{\ef}{\mathfrak{e}}

\DeclareMathOperator{\Gal}{Gal}

\newtheorem{theorem}{Theorem}
\newtheorem{lemma}[theorem]{Lemma}
\newtheorem{proposition}[theorem]{Proposition}
\newtheorem{corollary}[theorem]{Corollary}
\theoremstyle{definition}

\numberwithin{theorem}{section}
\numberwithin{equation}{section}
\numberwithin{table}{section}

\newcommand\ZZ{\mathbb{Z}}
\newcommand\NN{\mathbb{N}}
\newcommand\QQ{\mathbb{Q}}
\newcommand\RR{\mathbb{R}}
\newcommand\CC{\mathbb{C}}

\newcommand{\OO}{\mathcal{O}}

\newcommand{\id}[1]{\mathfrak{#1}}
\newcommand{\ppp}{\id{p}}
\newcommand{\PPP}{\id{P}}

\newcommand{\ddd}{\id{d}}

\newcommand{\absnorm}{\mathfrak{N}}

\newcommand\card{\#}

\DeclareMathOperator\Hom{Hom}

\newcommand{\abs}[1]{\left|#1\right|}

\newcommand{\F}{\mathfrak{F}}

\newcommand{\gext}{G\text{-ext}}

\newcommand\Fbar{\overline{F}}
\newcommand{\npid}{l}
\DeclareMathOperator{\Cl}{Cl}
\DeclareMathOperator{\Res}{Res}
\newcommand{\idele}{\mathbf{A^\times}}

\DeclareMathOperator\Aut{Aut}
\newcommand{\totram}{\mathcal{T}_{F,n}}

\newcommand{\U}{\mathcal{U}}

\newcommand{\SSS}{\mathcal{F}}

\begin{document}

\title{Average bounds for the $\ell$-torsion in class groups of cyclic extensions}

\author{Christopher Frei}
\address{School of Mathematics\\
University of Manchester\\
Oxford Road, Manchester M13 9PL\\
UK}
\email{christopher.frei@manchester.ac.uk}

\author{Martin Widmer}
\address{Department of Mathematics\\ 
Royal Holloway, University of London\\ 
TW20 0EX Egham\\ 
UK}
\email{martin.widmer@rhul.ac.uk}

\date{\today}

\dedicatory{Dedicated to Professor Robert F. Tichy on the occasion of his 60th birthday.}


\begin{abstract}
  For all positive integers $\ell$, we prove non-trivial bounds for the
  $\ell$-torsion in the class group of $K$, which hold for almost all
  number fields $K$ in certain families of cyclic extensions of
  arbitrarily large degree. In particular, such bounds hold for almost
  all cyclic degree-$p$-extensions of $F$, where $F$ is an arbitrary
  number field and $p$ is any prime for which $F$ and the $p$-th
  cyclotomic field are linearly disjoint. Along the way, we prove
  precise asymptotic counting results for the fields of bounded
  discriminant in our families with prescribed splitting behavior at
  finitely many primes.
\end{abstract}

\maketitle
\setcounter{tocdepth}{1}
\tableofcontents

\section{Introduction}\label{introductionchapter3}
Let $F$ be a number field with ring of integers $\OO_F$ and algebraic closure $\Fbar$. 
Let $n>1$ be an integer such that $F$ and $\QQ(\mu_n(\Fbar))$ are linearly
disjoint over $\QQ$, where $\mu_n(\Fbar)$ is the group of $n$-th roots of unity
in $\Fbar$. In this paper, we consider the family $\totram$ of Galois extensions $K\subseteq
\Fbar$ of $F$ with cyclic Galois group of order $n$ that satisfy the following
condition: 
\begin{quote}\label{familydef}
  every prime ideal of $\OO_F$ not dividing $n$ is either unramified or totally ramified in $K$.
\end{quote}
This condition is vacuous if $n$ is prime, so in this case
$\totram$ is the family of all cyclic degree-$n$-extensions of $F$. We
prove for each integer $\ell\geq 1$ an unconditional non-trivial upper
bound for the size of the $\ell$-torsion subgroup $Cl_K[\ell]$ of the
class group of $K$, which holds for all but a zero density set of
fields $K\in \totram$.  The case $n=2$ over the ground field $F=\IQ$
has been proven recently by Ellenberg, Pierce, and Wood \cite{EllenbergPierceWood}.

\subsection{Background}
We always assume $X\geq 2$, and that $\ell$ is a positive integer. We shall use
the $O(\cdot)$, $\ll$, and $\gg$ notation; unless explicitly stated otherwise,
the implied constants will depend only on the indicated parameters.
Denote the modulus of the discriminant of the number field $K$ by $\DK$, and
its degree $[K:\IQ]$ by $d$.

Bounding $\#Cl_K[\ell]$ by the size of the full class group, and using \cite[Thm 4.4]{Nark1980}  yields the trivial bound\footnote{As usual $\varepsilon$ denotes an arbitrarily small positive number.} 
\begin{alignat}1\label{trivial}
\#Cl_K[\ell] \ll_{d,\varepsilon} \DK^{1/2+\varepsilon},
\end{alignat}
valid for all number fields $K$ and positive integers $\ell$. A widely open conjecture (see, e.g., \cite[Conjecture 1.1]{EllVentorclass}, \cite[Section 3]{Duke} and \cite{Zhang05})  states that
\begin{alignat}1\label{epsilonconjecture}
\#Cl_K[\ell] \ll_{d,\ell,\varepsilon} \DK^{\varepsilon}.
\end{alignat}
For  $d=\ell=2$ the conjecture follows from Gauss' genus theory
but unconditional non-trivial bounds that hold for all number fields of degree $d$
are known only for $\ell=2$, and for $d\leq 4$ and $\ell=3$.
For $(d, \ell)=(2,3)$ the first non-trivial bounds were obtained by Pierce \cite{Pierce05, Pierce06}, and Helfgott and Venkatesh  \cite{HelfgottVenkatesh}. Currently 
the record is Ellenberg and Venkatesh's  \cite{EllVentorclass} bound
$$\#Cl_K[3] \ll_{\varepsilon} \DK^{1/3+\varepsilon}$$ 
which holds
also for cubic fields. They also established a non-trivial bound for quartic fields
(with, e.g., an exponent $83/168+\varepsilon$ provided $K$ is an $S_4$ or $A_4-$field).
More recently  Bhargava, Shankar, Taniguchi, Thorne, Tsimerman, and Zhao \cite{Bhargava2tor} showed for arbitrary $d$ the bound
$$\#Cl_K[2] \ll_{d,\varepsilon} \DK^{1/2-1/(2d)+\varepsilon},$$
and for $d\in \{3,4\}$ they can even take 
the exponent $0.2784$ (consequently giving the new world record upper bound for the number of $A_4$-extensions of bounded discriminant). Their method is entirely different and based on geometry of numbers, but unfortunately seems 
not to extend to any $\ell>2$. The most general non-trivial bound is due to Ellenberg and Venkatesh  \cite[Proposition 3.1]{EllVentorclass}
and states
\begin{alignat}1\label{GRHbound}
\#Cl_K[\ell] \ll_{d,\ell,\varepsilon} \DK^{1/2-1/(2\ell(d-1))+\varepsilon}.
\end{alignat}
This bound holds for all number fields $K$ and all $\ell$ but unfortunately is conditional on GRH. More generally, if $F$ is an arbitrary number field and $K/F$ is an extension of degree $n$ then, assuming GRH,  
\cite[Lemma 2.3]{EllVentorclass} provides the upper bound $\#Cl_K[\ell] \ll_{F,n,\ell,\varepsilon} \DK^{1/2-1/(2\ell(n-1))+\varepsilon}$;
in the sequel we shall refer to this as the ``GRH-bound''.

It is worthwhile mentioning that the Cohen-Lenstra heuristics (and their generalisations to arbitrary number fields by Cohen and Martinet \cite{CoMa}) predict the bound (\ref{epsilonconjecture})
but only for almost all fields $K$ of degree $d$ and certain ``good'' primes $\ell$.  
In this direction Ellenberg, Pierce and Wood \cite{EllenbergPierceWood} have 
shown that  the above GRH-bound (\ref{GRHbound}) holds unconditionally for almost all fields $K$ of degree $d$, at least for small degrees $d$ and sufficiently large $\ell$.

 \begin{theorem}[Ellenberg, Pierce, Wood]\label{ThmEPW}
Suppose $d\in \{2,3,4,5\}$,  $\nu_0(2)=\nu_0(3)=1$, $\nu_0(4)=1/48$, $\nu_0(5)=1/200$, and $\varepsilon>0$. Then for all but 
\begin{equation*}
O_{\ell,\varepsilon}(X^{1-\min\{1/(2\ell(d-1)),\nu_0(d)\}+\varepsilon})
\end{equation*}
degree $d$ number fields $K$
with $\DK\leq X$  (and non-$D_4$ when $d=4$) we have
\begin{alignat}1\label{EPWbound}
\#Cl_K[\ell] \ll_{\ell,\varepsilon} \DK^{1/2-\min\{1/(2\ell(d-1)),\nu_0(d)\}+\varepsilon}.
\end{alignat}
\end{theorem}
Note that the number of fields $K$ of degree $d\leq 5$  (and non-$D_4$ when $d=4$)  with $\DK\leq X$ grows linearly  in $X$, so that $100\%$ of these  fields satisfy the bound (\ref{EPWbound})
when enumerated by modulus of the discriminant. The cases $d=4,5$ were recently improved by the second author \cite{ltor}.

\subsection{Results}
Theorem \ref{ThmEPW} relies on uniform, power-saving error terms for the asymptotics of degree $d$-fields with chosen
splitting types at a finite set of primes; the results in \cite{EllenbergPierceWood}  are formulated in such a way that whenever 
such asymptotics are available 
then the GRH-bound (\ref{GRHbound}) will hold for almost all $K$ and all
$\ell$ sufficiently large. As  already mentioned in 
\cite[last paragraph of Section 2.]{EllenbergPierceWood}, this extends straightforwardly to allow arbitrary ground fields $F$.
Our first result combines this extension with an idea from \cite{ltor} to make further progress by showing that one can even go beyond the GRH-bound,
at least for $n>3$. 
Here we content ourselves with a simple consequence of Theorem \ref{thm:impGRH}.

Let $\SSS$ be a family of  
degree $n$ extensions $K\subseteq \Fbar$ of $F$, let $\E$ be a finite set of prime ideals
$\p$ in $\Oseen_F$, and set
\begin{equation*}
 N_\SSS(X):=\card\{K\in \SSS; \DK\leq X\}.
\end{equation*}
Let $\ef=\p$ or $\ef=\p\q$ for distinct prime ideals $\p$ and $\q$ of
$\Oseen_F$ with $\p,\q\notin\E$, and let $N_\SSS(\ef;X)$ be the number of
fields $K$ counted in $N_\SSS(X)$ in which the prime ideals dividing $\ef$ split
completely. Suppose that $c_\SSS >0$, $0\leq \tau< 1$, and $\sigma\geq 0$, and that we have 
\begin{alignat}1\label{IntroScounting1}
N_\SSS(X)&=c_\SSS X+ O_{\SSS,\varepsilon}\left(X^{\tau+\varepsilon}\right), \\ 
\label{IntroScounting2}
N_\SSS(\ef;X)&=\delta_\ef c_\SSS X+ O_{\SSS,\varepsilon}\left((\absnorm(\ef))^{\sigma}X^{\tau+\varepsilon}\right), 
\end{alignat}
where $\delta_\ef$ is a multiplicative function with
$1\ll_{\SSS}\delta_\p\leq 1$ if $\p\notin \E$.

\begin{theorem}\label{introthm:impGRH}
Suppose $F$ is a number field, and $\SSS$ is a family of  
degree $n$ extensions $K\subseteq \Fbar$ of $F$. Let $\varepsilon>0$, and suppose (\ref{IntroScounting1}) and (\ref{IntroScounting2}) do hold for the family $\SSS$ and some finite set $\E$ of primes in $\Oseen_F$.
Then we have for all sufficiently large $\ell$ 
\begin{alignat}1\label{impGRHbound}
\#Cl_K[\ell] \ll_{\SSS,\ell,\varepsilon} \DK^{1/2-\frac{1}{\ell(n+1)}+\varepsilon}
\end{alignat}
for $100\%$ of $K\in \SSS$ (when enumerated by modulus of the discriminant).
\end{theorem}

The hypotheses of Theorem \ref{introthm:impGRH} are believed to hold, for example, for the family of degree-$n$-extensions of $F$ whose normal closure has Galois group $S_n$. However, at present times this is known only for a few cases, for instance when $n\leq 5$ and $F=\QQ$ (see \cite{EllenbergPierceWood}).

Our first main result generalises the case $d=2$ of Theorem \ref{ThmEPW} in two different 
directions\footnote{The cases $\ell\leq 3$ are actually not covered but these cases are superseded by the aforementioned stronger pointwise bounds. And in fact our  
stronger Theorem \ref{thm:general}  fully covers the
case $d=2$ of Theorem \ref{ThmEPW} too.}. Write $m=[F:\QQ]$, and recall that 
$\totram$ is a certain family of cyclic degree-$n$-extensions of $F$. 
We define 
  \begin{equation}\label{def:deltatilde}
    \tilde{\delta}=\tilde{\delta}(m,n):=
    \begin{cases}
      \frac{1}{8\phi(n)(n-1)} &\text{ if
      }m=1\\
      \frac{1}{2(m+1)\phi(n)(n-1)} &\text{ if
      }m\geq 2,
    \end{cases}
  \end{equation}
  where $\phi(\cdot)$  denotes Euler's totient function.

\begin{theorem}\label{Thm3}
Suppose $F$ and $\QQ(\mu_n(\Fbar))$ are linearly disjoint over $\QQ$, and $\varepsilon>0$. Then for all but $O_{F,n,\varepsilon}(X^{\frac{1}{n-1}-\min\{\frac{1}{2\ell(n-1)},\tilde{\delta}\}+\varepsilon})$ fields $K$ in  $\totram$ with $\DK\leq X$ we have 
\begin{alignat}1\label{Thm3bound}
\#Cl_K[\ell] \ll_{F,n,\ell,\varepsilon} \DK^{\frac{1}{2}-\min\{\frac{1}{2\ell(n-1)},\tilde{\delta}\}+\varepsilon}.
\end{alignat}
\end{theorem}
Since the number of $K\in \totram$ with $\DK\leq X$ grows with the order
$X^{1/(n-1)}$ (cf. Theorem \ref{thm:count_quad} below), we conclude that when
enumerated by $\DK$ then $100\%$ of the fields $K\in\totram$ satisfy the bound
(\ref{Thm3bound}). Theorem \ref{Thm3} offers several ways of obtaining families
of number fields of arbitrarily large degree, for which non-trivial bounds for
$\card\Cl_K[\ell]$ are known for every given $\ell$, for almost all members
of the family. 

For which $\ell$ does the bound on the right-hand side of
(\ref{Thm3bound}) become the GRH-bound?  For $m\geq 2$, we get the
GRH-bound if and only if $\ell\geq m+1$. Theorem \ref{Thm3} will
follow from Theorem \ref{thm:general}, which holds for slightly more
general families of number fields and provides a slightly larger value
for $\tilde{\delta}$.

Partial summation, using the trivial bound (\ref{trivial}) for the exceptional fields, immediately gives the following average bound.
\begin{corollary}\label{Cor1}
Suppose $F$ and $\QQ(\mu_n(\Fbar))$ are linearly disjoint over $\QQ$, and $\varepsilon>0$.
Then we have 
$$\sum_{K\in \totram \atop \DK\leq X}\#Cl_K[\ell] \ll_{F,n,\ell,\varepsilon} X^{\frac{1}{2}+\frac{1}{n-1}-\min\{\frac{1}{2\ell(n-1)},\tilde{\delta}\}+\varepsilon}.$$
\end{corollary}

A result quite similar to Theorem \ref{Thm3} is obtained
independently by Pierce, Turnage-Butterbaugh and Wood as part of a
very recent preprint \cite{PTWarXiv}. The main differences seem to be
that they obtain much better bounds on the size of the exceptional
set, but their results are restricted to the base field $F=\QQ$, and
it seems unclear to which extent their arguments extend to other
number fields. Their proofs are based on an effective version of the
Chebotarev density theorem for certain families of number fields, for
which they also include several further applications. For example,
they also obtain conditional results for $S_n$-extensions of degree
$n$ and squarefree discriminant. It would be interesting to see
whether the ideas from our Section \ref{SecImpGRH} could be used
beneficially in their proofs to beat the GRH-bound.
  
Our second main result and crucial new input to establish Theorem \ref{Thm3}
is a counting result for the number of fields in the family $\totram$ of bounded
discriminant satisfying prescribed local conditions, with a fairly explicit error term.

We write $\Delta(K/F)$ for the relative discriminant ideal of the extension $K/F$ and $\absnorm\Delta(K/F)$ for its absolute norm. In our notation, we will
not distinguish between prime ideals of the ring of integers $\OO_F$
and the corresponding non-archimedean places of $F$.
For $\npid\geq 0$ and a set $\PPP=\{\ppp_1,\ldots,\ppp_\npid\}$ of pairwise
distinct non-archimedean places of $F$ not dividing $n$, we study the counting function
\begin{equation*}
  N_{\totram}(\PPP;X) := \card\left\{K \in \totram;\ 
    \ppp_1,\ldots,\ppp_\npid\text{ split completely in }K,\
    \absnorm\Delta(K/F)\leq X \right\}.
\end{equation*}

\begin{theorem}\label{thm:count_quad}
  Suppose $F$ and $\QQ(\mu_n(\Fbar))$ are linearly disjoint over $\QQ$, and $\varepsilon>0$. Then 
  \begin{equation*}
    N_{\totram}(\PPP; X) = \delta_{\ppp_1\cdots\ppp_l}c_{F,n}X^{1/(n-1)} + O_{F,n,\npid,\varepsilon}\left(\absnorm(\ppp_1\cdots\ppp_\npid)^{1/(2m)+\varepsilon}X^{(1-\beta)/(n-1)}\right), 
  \end{equation*}
  as $X\to\infty$, where 
  \begin{equation*}
    \beta:=
    \begin{cases}
      1/(4\phi(n)) &\text{ if }m=1,\\
      1/(2m\phi(n)) &\text{ if }m\geq 2.
    \end{cases}
  \end{equation*}
  The constant $c_{F,n}$ is positive and can be computed explicitly. The
  constant $\delta_\ddd$ is multiplicative\footnote{Of course, for $l=0$ we set $\ppp_1\cdots\ppp_l:=\OO_F$ and  
  $\delta_{\OO_F}:=1$.} in $\ddd$ and satisfies 
  \begin{equation}\label{eq:lead_const}
  \delta_\ppp:=\begin{cases}
    1/n &\text{ if }\absnorm(\ppp)\not\equiv 1\bmod n\\
    1/(n(1+\phi(n)\absnorm(\ppp)^{-1})) &\text{ if }\absnorm(\ppp)\equiv 1\bmod n.
  \end{cases}
\end{equation}
\end{theorem}

If $n$ is prime, a weaker version of Theorem \ref{thm:count_quad}, with an
error term of the shape $O_{F,n,\PPP}(X^{1/(n-1)-\gamma})$, for some $\gamma>0$,
follows from \cite[Theorem 1.7]{DanRachel}. It is crucial for our work here to
know the dependence of the error term on $\ppp_1,\ldots,\ppp_\npid$ explicitly,
as well as the value of $\gamma$. The multiplicativity of the constant
$\delta_\ddd$ can be interpreted as asymptotic independence of the local conditions
imposed at $\ppp_1,\ldots,\ppp_l$. This is required for our application in
Theorem \ref{Thm3} and constitutes the main reason for restricting our
attention to the family $\totram$. The family of all cyclic
degree-$n$-extensions would not show this independence behavior with respect to
local conditions when counted by discriminant, unless $n$ is prime.

Comparing our Theorem \ref{thm:count_quad} with \cite[Theorem
1.7]{DanRachel}, we observe that $\totram$ has density zero in the
family of all cyclic degree-$n$-extensions of $F$, unless $n$ is
prime. Not many other precise counting results for interesting
zero-density families of number field extensions with fixed Galois
group are available in the literature. A zero-density family of
biquadratic fields was recently considered in \cite{Nick}.

The hypothesis that $F$ be linearly disjoint from $\QQ(\mu_n(\Fbar))$
is necessary for the result to hold. If it is not satisfied, the
asymptotic formula for $N_{\totram}(\PPP;X)$ will involve logarithms
and secondary main terms. For prime $n$, this can be observed from
\cite[Theorem 1.7]{DanRachel}.

\subsection{Discriminant zeta function} We prove Theorem
\ref{thm:count_quad} by studying the corresponding zeta function,
using the approach of \cite{MR969545, DanRachel}. For locally compact
abelian groups $A$ and $B$, we denote by $\Hom(A,B)$ the group of
\emph{continuous} homomorphisms from $A$ to $B$, equipped with the
compact-open topology. Let $G=\mu_n$ be the group of $n$-th roots of
unity in $\CC$. We let $\gext(F)\subseteq\Hom(\Gal(\Fbar/F), G)$ be
the set of continuous \emph{surjective} homomorphisms
$\Gal(\Fbar/F)\to G$. A homomorphism $\varphi\in\gext(F)$ corresponds
uniquely to a pair $(K_{\varphi}/F,\psi)$, where $K_\varphi/F$ is a
Galois extension and $\psi$ is an isomorphism
$\Gal(K_\varphi/F)\to G$. Indeed, we take $K_\varphi$ to be the fixed
field of $\ker\varphi$ and $\psi$ the homomorphism induced by
$\varphi$ on the quotient $\Gal(K_\varphi/F)$. Clearly, each
$K_\varphi$ is induced by $\card\Aut(G)=\phi(n)$ different
$\varphi\in\gext(F)$. We write $\Delta(\varphi):=\Delta(K_\varphi/F)$.

For each place $v$, we fix an algebraic closure $\Fbar_v\supseteq\Fbar$. Then
each homomorphism $\varphi\in\Hom(\Gal(\Fbar/F), G)$ defines local homomorphisms
$\varphi_v\in\Hom(\Gal(\Fbar_v/F_v),G)$. Let $e(\varphi_v)$ denote the
ramification index of the corresponding local extension $K_{\varphi_v}/F_v$.
For non-archimedean $v$, we denote the cardinality of its residue field by
$q_v$.
 
Theorem \ref{thm:count_quad} will follow immediately from Theorem
\ref{thm:count_quad_general} below, which provides better error terms and
handles slightly more general families of fields, for which we also allow local
restrictions at places dividing $n\infty$.

For $v\mid n\infty$, let $\Lambda_v\subseteq \Hom(\Gal(\Fbar_v/F_v),G)$ be any
subset containing the trivial homomorphism $1$. With $\Lambda :=
(\Lambda_v)_{v\mid n\infty}$, we consider the family
\begin{equation*}
  \totram(\Lambda) := \{\varphi\in\gext(F);\ K_\varphi\in\totram\text{ and
  }\varphi_v\in\Lambda_v\text{ for all }v\mid n\infty\}
\end{equation*}
and the counting functions
\begin{equation*}
  N_{\totram}(\Lambda, \PPP; X) := \card\{\varphi\in\totram(\Lambda);\  
    \begin{aligned}
      \ppp_1,\ldots,\ppp_\npid\text{
        split completely in }K_\varphi,\ \absnorm\Delta(\varphi)\leq X
    \end{aligned}
   \}.
\end{equation*}

\begin{theorem}\label{thm:count_quad_general}
  Suppose $F$ and $\QQ(\mu_n(\Fbar))$ are linearly disjoint over $\QQ$, and
  $\varepsilon>0$. With $\Lambda$ as above, we have 
  \begin{equation*}
    N_{\totram}(\Lambda,\PPP; X) = \delta_{\ppp_1\cdots\ppp_l}c_{F,n,\Lambda}X^{1/(n-1)} + O_{F,n,\npid,\varepsilon}\left(\absnorm(\ppp_1\cdots\ppp_\npid)^{a+\varepsilon}X^{(1-b)/(n-1)+\varepsilon}\right), 
  \end{equation*}
  as $X\to\infty$, where 
  \begin{equation*}
    a:=
    \begin{cases}
      3/16 &\text{ if } n=2\text{ and }m=1,\\
      103/512 &\text{ if } n=2\text{ and }m=2,\\
      1/(2m) &\text{ if } n\geq 3\text{ or }m\geq 3,
    \end{cases}
  \end{equation*}
  and
  \begin{equation*}
    b:=
    \begin{cases}
      13/32 &\text{ if } n=2\text{ and }m=1,\\
      153/512 &\text{ if } n=2\text{ and }m=2,\\
      \min\{1/4, 64/(103\phi(n)m)\} &\text{ if }n\geq 3 \text{ or }m\geq 3.
    \end{cases}
  \end{equation*}
  The constant $c_{F,n,\Lambda}$ is positive and can be computed explicitly. The
  constant $\delta_\ddd$ is multiplicative in $\ddd$ and satisfies \eqref{eq:lead_const}.
\end{theorem}

With the choice $\Lambda_v := \Hom(\Gal(\Fbar_v/F_v),G)$ for all $v\mid n\infty$, we
get $N_{\totram}(\Lambda, \PPP; X) = \phi(n)N_{\totram}(\PPP;X)$. Thus, with
the observation that
\begin{equation}\label{eq:a_b_bounds}
  a\leq 1/(2m)\quad\text{ and }\quad b>
  \begin{cases}
    1/(4\phi(n)) &\text{ if }m=1,\\
    1/(2m\phi(n)) &\text{ if }m\geq 2,    
  \end{cases}
\end{equation}
one sees immediately that Theorem \ref{thm:count_quad_general} implies Theorem
\ref{thm:count_quad}. But our more general setup allows us to count other
interesting families as well, for example the family of all cyclic extensions
of degree $n$ of $F$, in which every tamely ramified prime ideal is totally
ramified, or the family in which every ramified prime ideal is totally ramified.

For the proof of Theorem \ref{thm:count_quad_general}, we define a function  $f(\PPP;\varphi):=\prod_vf_v(\PPP;\varphi_v)$ on $\Hom(\Gal(\Fbar/F), G)$
locally by
\begin{equation*}
  f_v(\PPP; \varphi_v):=
  \begin{cases}
    1&\text{ if }v\mid n\infty\text{ and }\varphi_v\in\Lambda_v,\\
    1&\text{ if }v\in\PPP\text{ and }\varphi_v=1,\\
    1&\text{ if }v\notin \PPP,\ v\nmid n\infty, \text{ and }e(\varphi_v)\in\{1,n\},\\
    0&\text{ otherwise.}
  \end{cases}
\end{equation*}
With this definition, $f(\PPP;\varphi)=1$ if and only if
$\varphi\in\totram(\Lambda)$ and $\ppp_1,\ldots,\ppp_l$ split completely in $K_\varphi/F$.  Thus, writing $\Delta(\varphi):=\Delta(K_\varphi/F)$, the
Dirichlet series corresponding to $N_{\totram}(\Lambda,\PPP; X)$ is
\begin{equation*}
  D(\Lambda,\PPP; s) := \sum_{\varphi\in\gext(F)}\frac{f(\PPP;\varphi)}{\absnorm(\Delta(\varphi))^s}.
\end{equation*}

\begin{proposition}\label{prop:ds_merom}
  The Dirichlet series $D(\Lambda,\PPP;s)$ converges absolutely in the half-plane $(n-1)\Re(s)>1$. It has a
  meromorphic continuation to the half-plane $(n-1)\Re(s)>1/2$. The only pole in this
  half-plane is a simple pole at $s=1/(n-1)$. The residue has the form
  $\delta_{\ppp_1\cdots\ppp_l}c_{F,n,\Lambda}(n-1)^{-1}$, as in Theorem \ref{thm:count_quad_general}. Let
  \begin{equation*}
    \alpha :=
    \begin{cases}
      3/8 &\text{ if }F=\QQ \text{ and }n=2\\
      103/256 &\text{ otherwise.}
    \end{cases}
  \end{equation*}
  Then, for any $\eta\in
  (0,1)$ and $\varepsilon>0$, we have the estimate
  \begin{equation}\label{eq:ds_estimate}
    \frac{|s-\frac{1}{n-1}|}{\abs{s}}D(\Lambda,\PPP;s) \ll_{F,n,l,\eta,\varepsilon}
    \left(\absnorm(\ppp_1\cdots\ppp_\npid)(1+|\Im s|)^{m}\right)^{\phi(n)\alpha(1+\eta-(n-1)\Re
        s)+\varepsilon}
  \end{equation}
  in the vertical strip $1/2+\varepsilon\leq(n-1)\Re(s)<1+\eta$.
\end{proposition}

We will first prove Proposition \ref{prop:ds_merom} with help of the techniques from
\cite{DanRachel} and then deduce Theorem \ref{thm:count_quad} from it via
Perron's formula and tauberian arguments.

\section{Proof of Proposition \ref{prop:ds_merom}: Analysis of the discriminant
zeta function}

The aim of this section is to prove Proposition \ref{prop:ds_merom}. In
\S\ref{sec:setup}, we apply class field theory, M\"obius inversion and a
version of the Poisson summation formula to express $D(\Lambda,\PPP;s)$ as a
sum of Euler products. The arguments are very similar to \cite{DanRachel}, so
we will be concise. In \S\ref{sec:local_factors}, we analyse these Euler
products and show that they behave like certain Artin $L$-functions. This will
yield a meromorphic continuation of $D(\Lambda,\PPP;s)$. In
\S\ref{sec:vertical}, we deduce \eqref{eq:ds_estimate} from subconvex bounds of
our Artin $L$-functions, and in \S\ref{sec:residue}, we determine the residue of
$D(\Lambda,\PPP;s)$ at $s=1/(n-1)$.

\subsection{Preliminaries}\ 

\begin{lemma}\label{lem:unram}
  Let $F$ be a number field, $n>1$ an integer and $S$ a finite set of places of $F$ containing all places dividing $n\infty$. Then the number of cyclic extensions $K/F$ with $[K:F]=n$ that are unramified at all places not in $S$ is $\ll_{F,n} n^{[F(\mu_n(\Fbar)):F]\cdot \card S}$.
\end{lemma}

\begin{proof}
  Consider first the case where $\mu_n(\Fbar)\subseteq F$. We may assume, without loss of generality, that $S$ contains enough places to ensure that the ring of $S$-integers $\OO_S$  of $F$ is a principal ideal domain.

  By Kummer theory, cyclic extensions $K/F$ of degree $n$ are of the form
  $K=F(\sqrt[n]{a})$, with $a\in F^\times/F^{\times n}$. For finite places
  $\ppp\notin S$, the extension $F(\sqrt[n]{a})/F$ is unramified if and only if
  $n$ divides the exponential $\ppp$-adic valuation $\ord_\ppp(a)$. Consider the exact sequence
  \begin{equation}\label{eq:quad_ext_exact}
    1\to \OO_{S}^\times \to F^\times \to \prod_{\ppp\notin S}\ZZ \to 1,
  \end{equation}
  where the third map is given by the exponential
  valuations at $\ppp\notin S$. Taking the tensor product with $\ZZ/n\ZZ$, we get
  \begin{equation*}
    \OO_{S}^\times/\OO_{S}^{\times n} \to F^\times/F^{\times n} \to
    \prod_{\ppp\notin S}\ZZ/n\ZZ\to 1,
  \end{equation*}
  so the number of possible values of $a\in F^\times/F^{\times n}$ is bounded by
  \begin{equation*}
    \card\OO_{S}^\times/\OO_{S}^{\times n} \ll_{F,n} n^{\card S}.
  \end{equation*}

  Now we consider the general case. For every $K$ as in the lemma, we get a
  cyclic extension $K(\mu_n(\Fbar))/F(\mu_n(\Fbar))$ of degree at most $n$,
  which is unramified at all places not lying above places in $S$. By what we
  proved above, the number of such extensions of $F(\mu_n(\Fbar))$ is
  $\ll_{F,n}n^{[F(\mu_n(\Fbar)):F]\cdot \card S}$. The lemma follows, since each
  cyclic extension field of $F(\mu_n(\Fbar))$ of degree bounded by $n$ has $\ll_{F,n} 1$ subfields.
\end{proof}

\subsection{Set-up and Poisson summation}\label{sec:setup}
We follow the strategy of \cite[\S 4]{DanRachel} with our
$f(\cdot)=f(\PPP;\cdot)$, additionally keeping track of the dependence on
$\ppp_1,\ldots,\ppp_\npid$. During this proof, all implicit constants in $O$-
and $\ll$-notation may always depend on $F,n,l$, but not on
$\ppp_1,\ldots,\ppp_l$. Since there are only finitely many possibilities for
$\Lambda$, once $n$ and $F$ are fixed, the implicit constants are also
independent of $\Lambda$.

Since $f(\PPP;\varphi)\leq 1$ for all $\varphi\in\gext(F)$, it is clear from Wright's result \cite{MR969545} that $D(\Lambda,\PPP;s)$ converges absolutely whenever $\Re s$ is large enough. 

Let $\idele$ be the idele group of $F$. Since $G=\mu_n$ is abelian,
global class field theory allows us to identify the groups
$\Hom(\Gal(\Fbar/F),G)$ and $\Hom(\idele/F^\times,G)$, and we
interpret $f$ as a function on the latter group. Via the natural
embedding $F_v^\times\subseteq\idele$, each
$\varphi\in\Hom(\idele/F^\times,G)$ induces local homomorphisms
$\varphi_v\in\Hom(F_v^\times,G)$ corresponding via local class field
theory to elements of $\Hom(\Gal(\Fbar_v/F_v),G)$. Thus, we may
describe the local factors of $f$ by
\begin{equation*}
  f_v(\PPP; \varphi_v)=
  \begin{cases}
    1&\text{ if }v\mid n\infty\text{ and }\varphi_v\in\Lambda_v,\\
    1&\text{ if }v\in\PPP\text{ and }\varphi_v=1,\\
    1&\text{ if }v\notin \PPP,\ v\nmid n\infty,\text{ and } [\OO_v^\times : \OO_v^\times \cap \ker(\varphi_v)]\in\{1,n\},\\
    0&\text{ otherwise.}
  \end{cases}
\end{equation*}

The following derivations are a direct application of \cite[\S 2]{DanRachel} to
the present situation. By the conductor-discriminant formula and M\"obius
inversion to remove the surjectivity condition, we get (cf. \cite[Lemma
2.2]{DanRachel})
\begin{equation}\label{eq:moebius}
  D(\Lambda,\PPP;s) = \sum_{d\mid n}\mu(n/d)F_{d}(\PPP;sn/d),
\end{equation}
where
\begin{equation*}
  F_d(\PPP;s):=\sum_{\varphi \in \Hom(\idele/F^\times,\  \mu_d)}\frac{f(\varphi)}{\Phi_d(\varphi)^{s}},
\end{equation*}
with
\begin{equation}
  \Phi_d(\varphi) := \prod_{a\bmod d}\Phi(\varphi^a),
\end{equation}
and $\Phi(\psi)$ the reciprocal of the idelic norm of the conductor of the character $\psi$.

If $d$ is a proper divisor of $n$ and $\varphi\in\Hom(\idele/F^\times,\  \mu_d)$, then $f(\varphi)=0$ unless the corresponding extension $K_\varphi$ is unramified at all places $v$ not dividing $n\infty$. By Lemma \ref{lem:unram}, this occurs for at most $\ll 1$ extensions. Thus, $F_d(\PPP;s)$ is a finite sum and entire. Moreover,
\begin{equation}\label{eq:Fd_bound}
  F_d(\PPP;s) \ll 1 \quad\text{ for } \quad\Re s\geq 0,\quad\text{ if }d\neq n.
\end{equation}

Thus, the analytic behavior of $D(\Lambda,\PPP;s)$ is determined by $F_n(\PPP;s)$.  We
let $S:=S' \cup \PPP$, where $S'$ is the set of places of $F$ dividing
$n\infty$.

By \cite[Proposition 3.8]{DanRachel}, a version of the Poisson
summation formula adapted to the present situation, the series $F_n(\PPP;s)$
has the form
\begin{equation}\label{eq:poisson}
  F_n(\PPP;s) = \frac{1}{\card\OO_F^\times/\OO_F^{\times n}}\sum_{x\in
    F^\times/F^{\times n}}\widehat{f}(x;s)\quad\text{ for }\Re s\gg 1,
\end{equation}
where $\widehat{f}(x;s)$ is a Fourier transform defined in \cite[\S
3.3]{DanRachel}. Denote by $x_v$ the image of $x\in F^\times/F^{\times n}$
under the natural map to $F_v^\times/F_v^{\times n}$. Using the observation
that $f_v(\PPP;\cdot)$ is invariant under $\Hom(F_v^\times/\OO_v^\times,\mu_n)$
for all $v\notin S$ and \cite[Lemma 3.6]{DanRachel}, we see that the sum in
\eqref{eq:poisson} extends in fact only over the finite group
\begin{equation*}
\U_S(n) := \{x\in F^\times/F^{\times n}\ : \ x_v\in
\OO_v^\times/\OO_v^{\times n} \text{ for all }v\notin S\}.
\end{equation*}
The Fourier transform $\widehat{f}(x;s)$ is an Euler product
$\widehat{f}(x;s)=\prod_v\widehat{f}_{v}(x_v;s)$, whose local factors at
$v\nmid n\infty$ we now describe explicitly. As in \cite[\S 3.3]{DanRachel}, we define for
$\chi_v\in\Hom(\OO_v^\times,\mu_n)$ and $x_v\in F_v^\times/F_v^{\times n}$ the average
\begin{equation*}
  \tau_{f_{v}}(\chi_v,x_v):=\frac{1}{n}\sum_{\psi_v\in\Hom(F_v^\times/\OO_v^\times,\
    \mu_n)}f_v(\PPP; \chi_v\psi_v)\psi_v(x_v).
\end{equation*}
Here we identified
$\Hom(F_v^\times,\mu_n)=\Hom(\OO_v^\times,\mu_n)\oplus\Hom(F_v^\times/\OO_v^{\times},\mu_n)$
by the choice of a uniformiser. By \cite[Lemma 3.3]{DanRachel}, we get the formula
\begin{equation}\label{eq:local_factor_formula}
  \widehat{f}_{v}(x_v;s) = \sum_{m\mid
    (n, q_v-1)}\left(\sum_{\substack{\chi_v\in\Hom(\OO_v^\times,\ \mu_n)\\\ker(\chi_v)=\OO_v^{\times m}}}\tau_{f_v}(\chi_v,x_v)\chi_v(x_v)\right)q_v^{-n(1-1/m)s},
\end{equation}
where we wrote $(n,q_v-1)$ for the greatest common divisor. For $v\in\PPP$, we clearly have
\begin{equation*}
  \tau_{f}(\chi_v,x_v)=
  \begin{cases}
    \frac{1}{n}&\text{ if }\chi_v=1,\\
    0&\text{ otherwise, }
  \end{cases}
\end{equation*}
and thus
\begin{equation}\label{eq:local_factor_p}
  \widehat{f}_{v}(x_v;s) =\frac{1}{n}\quad\text{ for }\quad v\in\PPP. 
\end{equation}
At the remaining places not dividing $n\infty$, we find the following situation.
\begin{lemma}\label{lem:local_factors_comp}
  Let $v\notin \PPP$, $v\nmid n$ be a non-archimedean place of $F$ and
  $x_v\in\OO_v^\times/\OO_v^{\times n}$. Let $d_v(x_v)$ be the largest divisor $d$
  of $n$, for which $x_v\in\OO_v^{\times d}/\OO_v^{\times n}$. Then
  \begin{equation*}
    \widehat{f}_v(x_v;s) =
    \begin{cases}
      1 & \text{ if } q_v\not\equiv 1\bmod n,\\
      1 + \left(\sum_{d\mid d_v(x_v)}\mu(n/d)d\right)q_v^{-(n-1)s} & \text{
        if }q_v\equiv 1\bmod n.
    \end{cases}
  \end{equation*}
\end{lemma}

\begin{proof}
The $\Hom(k_v^\times/\OO_v^\times,\ \mu_n)$-invariance of
$f_v(\PPP,\cdot)$ implies that $\tau_{f}(\chi_v,x_v)=f_v(\PPP; \chi_v)$. Recall
the definition of $f_v(\PPP,\cdot)$ and the fact that $[\OO_v^\times :
\OO_v^{\times m}] = [F(v)^\times : F(v)^{\times m}] = m$ whenever $m\mid q_v-1$, where $F(v)$ is the
residue field. Thus, we get from \eqref{eq:local_factor_formula} the identities
\begin{equation*}
  \widehat{f}_{v}(x_v;s) = 1\quad\text{ if }\quad n\nmid q_v-1
\end{equation*}
and
\begin{equation*}
    \widehat{f}_{v}(x_v;s) = 1 +
    \left(\sum_{\substack{\chi_v\in\Hom(\OO_v^\times,\ \mu_n)\\\ker\chi_v =
          \OO_v^{\times n}}}\chi_v(x_v)\right)q_v^{-n(1-1/n)s}\quad \text{ if
    }\quad n\mid q_v-1.
  \end{equation*}
To see that the sum inside the parentheses has the desired shape, use
inclusion-exclusion and the fact that, for $d\mid n$,
\begin{equation*}
  \sum_{\chi_v\in\Hom(\OO_v^\times/\OO_v^{\times d},\ \mu_n)}\chi_v(x_v) =
  \begin{cases}
    \card\Hom(\OO_v^\times/\OO_v^{\times d},\mu_n)=d & \text{ if
    }x_v\in\OO_v^{\times d}/\OO_v^{\times n},\\
    0 & \text{ otherwise. } 
  \end{cases}
\end{equation*}
\end{proof}
The lemma shows in particular that the Euler product defining
$\widehat{f}(x;s)$ converges absolutely and defines a holomorphic function in
the half-plane $(n-1)\Re(s)>1$. The same holds thus for the Dirichlet series
$F_n(\PPP;s)$ and $D(\Lambda,\PPP;s)$, since they have non-negative coefficients.

\subsection{Analysis of the local factors}\label{sec:local_factors}
We will now compare the local factors $\widehat{f}_v(x_v;s)$ to the local
factors of certain Artin $L$-functions. Let $F_0:=F(\mu_n(\Fbar))$. Our
hypotheses on $F$ imply that $[F_0:F]=\phi(n)$. For $x\in \U_S(n)$, we
choose a representative $a\in F^\times$ of $x$ and an $n$-th root $\alpha\in
\Fbar$ of $a$. We consider the field
\begin{equation*}
  F_{x} := F_0(\alpha) = F(\mu_n(\Fbar), \alpha),
\end{equation*}
which is clearly independent of the choice of $a$ and $\alpha$. Since
$\mu_n(\Fbar)\subseteq F_0$, we see that $F_x/F_0$ is cyclic of degree dividing
$n$.

Let $v\notin S$ with $q_v\equiv 1\bmod n$, so $F_v$ has primitive $n$-th roots
of unity and $v$ splits completely in $F_0$. Let $w$ a place of $F_0$ above
$v$, then $w$ is unramified in $F_x$, and we denote its Frobenius automorphism
by $\sigma_w\in\Gal(F_x/F_0)$.

Let $n_v(x) := n/d_v(x_v)$, then $n_v(x)$ is the smallest positive integer with
the property that $\alpha^{n_v(x)}\in F_v=F_{0,w}$, so
$F_{0,w}(\alpha)/F_{0,w}$ is cyclic of degree $n_v(x)$. In particular, the
order of $\sigma_w\in\Gal(F_x/F_0)$ is $n_v(x)$.

\begin{lemma}\label{lem:frobenius}
  Let $v\notin S$ with $q_v\equiv 1\bmod n$. Let $\sigma\in\Gal(F_x/F_0)$ be of
  order $n_v(x)$. Then there are exactly $\phi(n)/\phi(n_v(x))$ places $w\mid
  v$ of $F_0$ with $\sigma_w = \sigma$. 
\end{lemma}

\begin{proof}
  Write $G:=\Gal(F_x/F)$ and $A:=\Gal(F_x/F_0)$. Let $n':=[F_x:F_0]$. Then
  $A\cong \ZZ/n'\ZZ$ is a normal subgroup of $G$, so $G$ acts on $A$ by
  conjugation. Fix a primitive $n$-th root of unity $\zeta\in F_0$, then any
  $\sigma\in\Gal(F_x/F)$ is determined by $\sigma(\zeta)$ and
  $\sigma(\alpha)$. Let $\sigma\in\Gal(F_x/F_0)$, then
  $\sigma(\alpha)=\zeta^{an/n'}\alpha$, for $a\in\ZZ/n'\ZZ$. For any
  $b\in (\ZZ/n\ZZ)^\times$, there are $n'$ automorphisms $\tau\in G$ with
  $\tau(\zeta)=\zeta^b$. For any such $\tau$, we get
  $\tau\sigma\tau^{-1}(\alpha) = \zeta^{ban/n'}\alpha$, so
  $\tau\sigma\tau^{-1}=\sigma^b$. Thus, the orbit of $\sigma$ under conjugation
  by $G$ is the set of all $\sigma'\in A$ with order $|\sigma'|=|\sigma|$, and
  the stabilizer has order $n'\phi(n)/\phi(|\sigma|)$.

  The group $G$ also acts transitively on the prime ideals of $F_0$ above $v$
  via $w\mapsto \tau(w)$, and for the corresponding Frobenius elements we have
  $\tau\sigma_w\tau^{-1}=\sigma_{\tau(w)}$. Since $v$ splits completely, the
  stabilizer of any $w$ is $A$ of order $n'$. This shows that every $\sigma\in
  A$ of order $n_w(x)$ is the Frobenius element
  $\sigma=\tau\sigma_w\tau^{-1}=\sigma_{\tau(w)}$ for precisely
  $\phi(n)/\phi(n_v(x))$ different places $\tau(w)$ above $v$.
\end{proof}

 Since $\Gal(F_x/F_0)$ is cyclic, the same holds for its character group. For
 any character of full order, we have the following identity.

\begin{lemma}
  Let $v\notin S$ with $q_v\equiv 1\bmod n$. Let
  $\chi\in\Hom(\Gal(F_x/F_0),\CC^\times)$ be a character of order $|\chi|=[F_x:F_0]$. Then
  \begin{equation}\label{eq:character_sum}
    \sum_{w\mid v}\chi(\sigma_w) = \sum_{d\mid d_v(x_v)}\mu(n/d)d.
  \end{equation}
  The sum on the left-hand side runs over all places $w$ of $F_0$ above $v$.
\end{lemma}

\begin{proof}
  Write $n'=[F_x:F_0]$ and $n_v=n_v(x)$, so that $\Gal(F_x/F_0)\cong \ZZ/n'\ZZ$. Using
  Lemma \ref{lem:frobenius}, inclusion-exclusion and character orthogonality, we obtain
  \begin{equation*}
    \sum_{w\mid v}\chi(\sigma_w) =
    \frac{\phi(n)}{\phi(n_v)}\sum_{\substack{a\in\ZZ/n'\ZZ\\|a|=n_v}}\chi(a)=\frac{\phi(n)}{\phi(n_v)}\sum_{d\mid n_v}\mu(n_v/d)\sum_{\substack{a\in\ZZ/n'\ZZ\\|a|\mid d}}\chi(a) = \frac{\phi(n)}{\phi(n_v)}\mu(n_v).
  \end{equation*}
  To show that the last expression equals the right-hand side of
  \eqref{eq:character_sum}, we recall that $n_v = n/d_v(x_v)$. If
  $v_p(n)>v_p(d_v(x_v))+1$ for some prime $p$, then both expressions are
  zero.

  Thus, let us assume that $v_p(n)\in\{v_p(d_v(x)), v_p(d_v(x))+1\}$ for
  all primes $p$. We group together all prime factors $p$ of $n$, for which
  $v_p(n)=v_p(d_v(x))$ by writing $n=mD$, $d_v(x)=fD$, with $(m,D)=(f,D)=1$ and
  $v_p(m)=v_p(f)+1$ for all primes $p$. The right-hand side in
  \eqref{eq:character_sum} is then equal to
  \begin{equation*}
    \phi(D)\sum_{d\mid f}\mu(m/d)d=\phi(D)\mu(m/f)f=\phi(D)f\mu(n_v) = \frac{\phi(n)}{\phi(n_v)}\mu(n_v).
  \end{equation*}
  The last equality holds, since $\phi(n) = \phi(D)\phi(m) = \phi(D)\phi(m/f)f$.
\end{proof}

  For any character $\chi$ of $\Gal(F_x/F_0)$, we consider the Artin $L$-function
  \begin{equation*}
    L(F_x/F_0,\chi,s) = \prod_{w}\frac{1}{1-\chi(\sigma_w)q_w^{-s}},
  \end{equation*}
  the product running over all places $w$ of $F_0$ that are unramified in
  $F_x$. For a place $v\notin S$ of $F$, the local factor of
  $L(F_x/F_0,\chi,s)$ at $v$ is
  \begin{equation*}
    L_v(F_x/F_0,\chi,s)=\prod_{w\mid v}\frac{1}{1-\chi(\sigma_w)q_w^{-s}}.
  \end{equation*}

  \begin{lemma}\label{lem:loc_fac_artin}
    Let $\chi$ be a character of $\Gal(F_x/F_0)$ of order $[F_x:F_0]$
    and $v\notin S$. Then
    \begin{equation*}
      \widehat{f}_v(x_v;s)=L_v(F_x/F_0,\chi,(n-1)s) + O(q_v^{-2(n-1)s})\quad\text{
        for }\quad(n-1)\Re(s)>1/2.  
    \end{equation*}
  \end{lemma}

  \begin{proof}
    For $\Re(s)>1/2$, we have
    \begin{equation*}
      L_v(F_x/F_0,\chi,s)=1 + \sum_{w\mid v}\chi(\sigma_w)q_w^{-s} + O(q_v^{-2s}).
    \end{equation*}
    If $q_v\not\equiv 1\bmod n$, then $v$ does not split completely in $F_0$,
    so $q_w\geq q_v^2$ for all $w\mid v$ and $L_v(F_x/F_0,\chi,s)=1 +
    O(q_v^{-2s})$. If $q_v\equiv 1\bmod n$, then $q_w=q_v$ and, by
    \eqref{eq:character_sum},
     \begin{equation*}
      L_v(F_x/F_0,\chi,s)=1 + \left(\sum_{d\mid d_v(x_v)}\mu(n/d)d\right)q_v^{-s} + O(q_v^{-2s}).
    \end{equation*}
    Compare these expressions to Lemma \ref{lem:local_factors_comp}.
  \end{proof}

  \begin{lemma}\label{lem:meromorphic_cont}
    Let $\chi$ be a character of $\Gal(F_x/F_0)$ of order $[F_x:F_0]$ and $\varepsilon>0$. There is
    a holomorphic function $g(x;s)$ on $(n-1)\Re(s)>1/2$, satisfying
    $g(x;s)\ll_\varepsilon 1$ on $(n-1)\Re(s)\geq 1/2+\varepsilon$, such that
    \begin{equation*}
      \widehat{f}(x;s) = g(x;s)L(F_x/F_0,\chi,(n-1)s)\quad\text{ on }\quad(n-1)\Re(s)>1.
    \end{equation*}
  \end{lemma}

  \begin{proof}
    Let $(n-1)\Re(s)>1$. Lemma \ref{lem:loc_fac_artin} implies that
    \begin{equation*}
      \prod_{v\notin S}\widehat{f}_v(x_v;s) = h(x;s)L(F_x/F_0,\chi,(n-1)s),
    \end{equation*}
    with a holomorphic function $h(x;s)$ on $(n-1)\Re(s)>1/2$ satisfying
    $1\ll_\varepsilon h(x;s)\ll_\varepsilon 1$ on
    $(n-1)\Re(s)\geq 1/2+\varepsilon$. Muptiply $h(x;s)$ by the $|S|\ll 1$ local factors
    $\widehat{f}_v(x_v;s)$ at $v\in S$ to get $g(x;s)$. The upper bound
    for $g(x;s)$ remains intact, since $\widehat{f}_v(x_v;s)\ll 1$ for
    $\Re(s)\geq 0$ and all $v\in S$. 
  \end{proof}

  From Lemma \ref{lem:meromorphic_cont}, we obtain a meromorphic
  continuation of $\widehat{f}(x;s)$, and thus also of $F_n(\PPP;s)$
  and $D(\Lambda,\PPP;s)$, to $(n-1)\Re s>1/2$. Since $\Gal(F_x/F_0)$
  is abelian, the Artin $L$-function $L(F_x/F_0,\chi,s)$ is a Hecke
  $L$-function and as such entire whenever $\chi$ is
  non-trivial. Thus, the only possible pole of $D(\Lambda,\PPP;s)$ in
  $(n-1)\Re s>1/2$ is a simple pole at $s=1/(n-1)$, coming from
  $\widehat{f}(x;s)$ for $x\in\U_S(n)$ with $F_x=F_0$, in which case
  $L(F_x/F_0,\chi,s)$ is the Dedekind zeta function $\zeta_{F_0}(s)$
  of $F_0$.

\subsection{Estimates in vertical strips}\label{sec:vertical}
Let us prove next the estimate \eqref{eq:ds_estimate} for $D(\Lambda,\PPP;s)$ in the
vertical strip. We use the best available subconvexity bounds in conjunction
with the Phragmen-Lindel\"of principle to estimate $L(F_x/F_0,\chi,s)$. As
$F_x/F_0$ is abelian, this coincides with the Hecke $L$-function $L(\psi,s)$ of
some Dirichlet character of $F_0$ whose conductor $\mathfrak{f}_\psi$ divides
the conductor of $F_x/F_0$. Since only places $w$ of $F_0$ above places $v\in
S$ can ramify in $F_x$, and since places above $\ppp_1,\ldots,\ppp_l$ are at
worst tamely ramified, we get $\absnorm(\mathfrak{f}_\psi)\ll
\absnorm(\ppp_1\cdots\ppp_l)^{\phi(n)}$.

First assume that $[F_x:F_0]>1$. Then $\chi$ is non-trivial, so $\psi$ is a
non-principal Dirichlet character. For $t\in\RR$, we consider the twist $\psi_t
:= \psi\abs{\cdot}^{it}$, where $\abs{\cdot}$ is the norm character on the
idele class group of $F_0$. Then $\psi_t$ is a Hecke character of the same
conductor $\mathfrak{f}_\psi$ and analytic conductor $C(\psi_t)\leq
\absnorm(\mathfrak{f}_\psi)(2+|t|)^{m\phi(n)}$.

  From \cite[Theorem 1.1]{Wu2016}, applied to $\psi_t$, we get
  \begin{equation}\label{eq:burgess}
    L(1/2+it;\psi) \ll_{\varepsilon} \left(\absnorm(\mathfrak{f}_\psi)(1+|t|)^{m\phi(n)}\right)^{\alpha/2+\varepsilon},
  \end{equation}
  where
  \begin{equation*}
    \alpha = \frac{1}{2}- \frac{1-2\theta}{8},
  \end{equation*}
  with $\theta$ any exponent towards the Ramanujan-Petersson conjecture. By
  \cite{MR2811610}, one may take $\theta=7/64$. If $F=\QQ$ and $n=2$, then
  $L(s;\psi)$ is a Dirichlet $L$-function and we get from \cite{MR576334} the
  bound \eqref{eq:burgess} with $\alpha=3/8$.

  For any $\gamma\in(0,1)$, the previous bounds in conjunction with the Phragmen-Lindel\"of principle yield 
  \begin{equation*}
    L(F_x/F_0,\chi,s)=L(s;\psi)\ll_{\eta,\varepsilon}\left(\absnorm(\mathfrak{f}_\psi)(1+|\Im(s)|)^{m\phi(n)}\right)^{\alpha(1+\eta-\Re(s))+\varepsilon}  
  \end{equation*}
  in the strip $1/2\leq \Re(s)\leq 1+\eta$. By a similar argument, we find in
  case $F_x=F_0$ that
  \begin{equation*}
    \frac{|s-1|}{|s|}L(F_x/F_0,\chi,s) = \frac{|s-1|}{|s|}\zeta_{F_0}(s) \ll_{\eta,\varepsilon}(1+|\Im(s)|)^{m\phi(n)\alpha(1+\eta-\Re(s))+\varepsilon}. 
  \end{equation*}
  Together with Lemma \ref{lem:meromorphic_cont} and the fact that
  $\card\OO_S^\times/\OO_S^{\times n}\ll 1$, these observations are enough to
  prove the estimate \eqref{eq:ds_estimate}.

\subsection{Residues}\label{sec:residue}
It remains to show that the $D(\Lambda,\PPP; s)$ does indeed have a pole at $s=1/(n-1)$ and
to compute the residue. We have already seen that only the summands
$\widehat{f}(x; s)$ with $F_x=F_0$ contribute to the residue. Recall that
$S=S'\cup\PPP$, with $S'$ the set of places of $F$
dividing $n\infty$. 

\begin{lemma}\label{lem:USprime}
  Let $x\in\U_S(n)$ with $F_x=F_0$. Then $x\in\U_{S'}(n)$.
 \end{lemma}

\begin{proof}
  Let $a\in F^\times$ be any representative of $x$. Since $F_x=F_0$, there is
  $\alpha\in F_0$ with $\alpha^n=a$. Let $v\notin S'$ be any place of $F$. To
  prove the lemma, we must show that the $v$-adic valuation $\ord_v(a)$ is in
  $n\ZZ$. Since $F_0/F$ is unramified outside $S'$, we see that
  $F_v(\alpha)/F_v$ is unramified as well. Thus, with $w$ the extension of $v$
  to $F_v(\alpha)$,
  \begin{equation*}
    \ord_v(a) = \ord_{w}(a) =
    n\cdot\ord_{w}(\alpha)\in n\ZZ.
  \end{equation*}
\end{proof}

Moreover, if $F_x=F_0$, almost all the local factors $\widehat{f}_v(x_v;s)$ are
independent of $x$.

\begin{lemma}\label{lem:local_factors_indep}
  Let $x\in\mathcal{U}_{S'}(n)$ with $F_x=F_0$ and $v\notin S'$. Then
  \begin{equation*}
    \widehat{f}_v(x_v;s)=
    \begin{cases}
      \frac{1}{n} &\text{ if }v\in\PPP,\\
      1 &\text{ if }v\notin \PPP\text{ and }q_v\not\equiv 1\bmod n,\\
      1+\phi(n)q_v^{-(n-1)s} &\text{ if }v\notin \PPP\text{ and }q_v\equiv 1\bmod n.
    \end{cases}
  \end{equation*}
\end{lemma}

\begin{proof}
  For $v\in\PPP$, this is \eqref{eq:local_factor_p}. For $v\notin\PPP$ with
  $q_v\not\equiv 1\bmod n$, it is Lemma \ref{lem:local_factors_comp}. In the
  remaining case, it also follows from Lemma \ref{lem:local_factors_comp},
  since then $F_x=F_0$ implies that $d_v(x_v)=n$. Indeed, let $a\in F^\times$
  be a representative of $x$ and $\alpha\in F_0$ with $\alpha^n=a$. Since
  $q_v\equiv 1\bmod n$, we get $F_{0,w}=F_v$ for any place $w$ of $F_0$ above
  $v$, so $\alpha\in F_{0,w}=F_v$, and thus $a\in F_v^{\times n}$.
\end{proof}

With $\alpha:=n-1$ and $\zeta_{F_0,v}(s)$ defined as the product of all local
factors of $\zeta_{F_0}(s)$ at places above $v$, the residue of $D(\Lambda,\PPP;s)$ at
$1/(n-1)$ has the form $c(\Lambda,\PPP)\Res_{s=1/\alpha}\zeta_{F_0}(\alpha s)$, where
$c(\Lambda,\PPP) := \lim_{s\to 1/\alpha}\zeta_{F_0}(\alpha s)^{-1}D(\Lambda,\PPP;s)$. Using
Lemma \ref{lem:USprime} and Lemma \ref{lem:local_factors_indep}, we compute
$c(\Lambda,\PPP)$ as 
\begin{align*}
  c(\Lambda,\PPP) &= \lim_{s\to 1/\alpha}\zeta_{F_0}(\alpha s)^{-1}F_n(\PPP;s) =
  \frac{1}{|\OO_F^\times/\OO_F^{\times n}|}\sum_{x\in\mathcal{U}_{S}(n)}\lim_{s\to 1/\alpha}\zeta_{F_0}(\alpha
  s)^{-1}\widehat{f}(x;s)\\
&= \frac{1}{|\OO_F^\times/\OO_F^{\times n}|}\sum_{\substack{x\in\mathcal{U}_{S'}(n)\\F_x=F_0}}\prod_{v}\frac{\widehat{f}_{v}(x_v;1/\alpha)}{\zeta_{F_0,v}(1)}\\
&=\prod_{v\in\PPP}\frac{1}{n\zeta_{F_0,v}(1)}\prod_{\substack{v\notin\PPP\\v\nmid
  n\infty\\q_v\not\equiv 1\bmod n}}\frac{1}{\zeta_{F_0,v}(1)}\prod_{\substack{v\notin\PPP\\v\nmid
  n\infty\\q_v\equiv 1\bmod
  n}}\frac{1+\varphi(n)q_v^{-1}}{\zeta_{F_0,v}(1)}\\
&\hspace{5.3cm}\cdot\frac{1}{|\OO_F^\times/\OO_F^{\times n}|}\sum_{\substack{x\in\mathcal{U}_{S'}(n)\\F_x=F_0}}
\prod_{v\mid n\infty}\frac{\widehat{f}_v(x_v,1/\alpha)}{\zeta_{F_0,v}(1)}\\
&=\delta_{\ppp_1}\cdots\delta_{\ppp_l}\cdot c(\Lambda,\emptyset),
\end{align*}
with $\delta_{\ppp}$ as in \eqref{eq:lead_const}. Thus, we identify the residue
as
\begin{equation*}
  \delta_{\ppp_1}\cdots\delta_{\ppp_l}\frac{c_{F,n,\Lambda}}{n-1},
\end{equation*}
with $c_{F,n,\Lambda} := (n-1)c(\Lambda,\emptyset)\cdot \Res_{s=1/\alpha}\zeta_{F_0}(\alpha
s)$. 

To finish the proof of Proposition \ref{prop:ds_merom}, we must show that
$c(\Lambda,\emptyset)>0$, so $D(\Lambda,\PPP;s)$ does indeed have a pole at
$s=1/(n-1)$. We accomplish this first in the case $\Lambda = \Lambda_0 :=
(\{1\})_{v\mid n\infty}$, where all places $v\mid n\infty$ are required to
split completely in $K_\varphi/K$. In this case, one sees immediately from the
definition of the local fourier transform in \cite[\S 3.3]{DanRachel} that
\begin{equation*}
  \widehat{f}_v(x_v;s) =
  \begin{cases}
    1 &\text{ if }v\mid\infty\\
    1/n &\text{ if }v\mid n.
  \end{cases}
\end{equation*}
Since moreover $x=1\in\mathcal{U}_{S'}(n)$ satisfies $F_x=F_0$, we conclude that
\begin{equation*}
  \sum_{\substack{x\in\mathcal{U}_{S'}(n)\\F_x=F_0}}
\prod_{v\mid n\infty}\frac{\widehat{f}_v(x_v,1/\alpha)}{\zeta_{F_0,v}(1)} >0,
\end{equation*}
which is enough to show that $c(\Lambda_0,\emptyset)>0$. For general $\Lambda$,
we have $D(\Lambda,\PPP;s) \geq D(\Lambda_0,\PPP;s)$ for all $s>1/(n-1)$, and
thus $D(\Lambda,\PPP;s)$ must also have a pole at $s=1/(n-1)$.

\section{Proof of Theorem \ref{thm:count_quad_general}: Contour integration}
In this section, we deduce Theorem \ref{thm:count_quad_general} from
Proposition \ref{prop:ds_merom}.

\subsection{Preliminaries and set-up}
We need the following lemma to estimate the coefficients of a Dirichlet
series. Still, all implied constants may always depend on $F,n,\npid$.

\begin{lemma}\label{lem:finite_quad}
  For $m\in\NN$, let 
  \begin{equation*}
b_m:=\{K/F\text{ cyclic; }\ [K:F]=n\ \text{ and }\  \absnorm(\Delta(K/F))=m\}.
\end{equation*}
 Then
  $b_m\ll_{\varepsilon} m^\varepsilon$ holds for all $\varepsilon>0$.
\end{lemma}

\begin{proof}
  Let $S$ be the set of all places of $F$ dividing $mn\infty$. Every extension
  $K/F$ with $\absnorm(\Delta(K/F))= m$ is unramified outside $S$. From Lemma
  \ref{lem:unram}, we deduce the bound $b_m \ll n^{\phi(n)|S|}\ll_\varepsilon m^\varepsilon$. 
\end{proof}

Write $a_m:=\card\{\varphi\in\gext(F); \ \absnorm(\Delta(\varphi))=m,\
f(\PPP;\varphi)=1\}$. Then
\begin{equation*}
  N_{\totram}(\Lambda,\PPP;X)=\sum_{m\leq X}a_m\quad \text{ and }\quad D(\Lambda,\PPP;s)=\sum_{m\in\NN}\frac{a_m}{m^s}. 
\end{equation*}

\subsection{The case $n=2$}
Let us first consider the case $n=2$. Choose $\eta\in(0,1/2)$,
$\sigma_0\in(1,1+\eta)$ and $T\in [1,X]$. By a truncated version of Perron's
formula (\cite[Corollary 5.3]{MR2378655}), we see that
\begin{align}
  N_{\totram}(\Lambda,\PPP;X) &- \frac{1}{2\pi
    i}\int_{\sigma_0-iT}^{\sigma_0+iT}D(\Lambda,\PPP;s)\frac{X^s}{s}\mathrm{d} s\label{eq:perron}\\\nonumber &\ll
  \sum_{X/2\leq m\leq 2X}a_m\min\left\{1,\frac{X}{T|X-m|}\right\} + \frac{4^{\sigma_0}+X^{\sigma_0}}{T}\sum_{m\in\NN}\frac{a_m}{m^{\sigma_0}}. 
\end{align}
Using Lemma \ref{lem:finite_quad} to estimate $a_m\ll_{\varepsilon}m^\varepsilon$ and replacing the minimum by its second term
unless $|X-m|<1$, the first error term is
\begin{equation}\label{eq:error_1}
  \ll_{\varepsilon}X^{\varepsilon}\left(1+\frac{X}{T}(\log X)\right)\ll_{\varepsilon} \frac{X^{1+2\varepsilon}}{T}.
\end{equation}
For $\varepsilon\in (0,\sigma_0-1)$, we see that the second error term is
\begin{equation}\label{eq:error_2}
  \ll_{\varepsilon,\sigma_0} \frac{X^{\sigma_0}}{T}.
\end{equation}
Recall the analytic facts about $D(\Lambda,\PPP;s)$ from Proposition
\ref{prop:ds_merom}. Shifting the line of integration to $\Re s = \sigma_1:=1-\lambda$, for
some $\lambda\in (0,1/2-\eta]$, we see that the integral in \eqref{eq:perron} equals
\begin{equation}\label{eq:residue_thm}
  2\pi i\Res_{s=1}\left(D(\Lambda,\PPP;s)\frac{X^s}{s}\right) + \left(-\int_{\sigma_1-iT}^{\sigma_0-iT}+\hspace{-0.1cm}\int_{\sigma_1-iT}^{\sigma_1+iT}+\hspace{-0.1cm}\int_{\sigma_1+iT}^{\sigma_0+iT}\right)D(\Lambda,\PPP;s)\frac{X^s}{s}\mathrm
  d s .
\end{equation}
Here,
\begin{equation*}
  \Res_{s=1}\left(D(\Lambda,\PPP;s)\frac{X^s}{s}\right) = X\Res_{s=1}(D(\Lambda,\PPP;s)) = \delta_{\ppp_1\cdots\ppp_l}c_{F,n}X
\end{equation*}
is the main term in our asymptotic expansion of $N_{\totram}(\PPP;X)$. Let us estimate
the integrals in \eqref{eq:residue_thm} from above. Using the estimate \eqref{eq:ds_estimate}, we see that
\begin{equation*}
  D(\Lambda,\PPP;s)\ll_{\eta,\varepsilon}(\absnorm(\ppp_1\cdots\ppp_\npid)T^m)^{\alpha(\eta+\lambda)+\varepsilon} \text{
  for }\ s\in(\sigma_1\pm iT, \sigma_0\pm iT),
\end{equation*}
and thus
\begin{equation}\label{eq:error_3}
  \int_{\sigma_1\pm iT}^{\sigma_0\pm iT}D(\Lambda,\PPP;s)\frac{X^s}{s}\mathrm d s \ll_{\eta,\varepsilon,\sigma_0,\lambda}\absnorm(\ppp_1\cdots\ppp_\npid)^{\alpha(\eta+\lambda)+\varepsilon}T^{m\alpha(\eta+\lambda)+\varepsilon-1}X^{\sigma_0}.
\end{equation}
Moreover,
\begin{align}
  \int_{\sigma_1-iT}^{\sigma_1+iT}D(\Lambda,\PPP;s)\frac{X^s}{s}\mathrm d s &\ll_{\eta,\varepsilon,\lambda}
  X^{1-\lambda}\absnorm(\ppp_1\cdots\ppp_\npid)^{\alpha(\eta+\lambda)+\varepsilon}\left(1+\int_{t=1}^Tt^{m\alpha(\eta+\lambda)+\varepsilon-1}\mathrm
    d t\right)\nonumber\\\label{eq:error_4}
&\ll_{\eta,\lambda,\varepsilon} X^{1-\lambda}\absnorm(\ppp_1\cdots\ppp_\npid)^{\alpha(\eta+\lambda)+\varepsilon}T^{m\alpha(\eta+\lambda)+\varepsilon}.
\end{align}
To optimise our error terms in case $m\geq 3$, we choose $\varepsilon$ small, $\eta=4\varepsilon$,
$\sigma_0=1+2\varepsilon$, $\lambda=1/(2m\alpha)$ and $T=X^{\lambda}$. Then the sum of the error terms
\eqref{eq:error_1}--\eqref{eq:error_4} is
\begin{equation*}
  \ll_{\varepsilon}\absnorm(\ppp_1\cdots\ppp_\npid)^{1/(2m)+3\varepsilon}X^{1-64/(103m)+5\varepsilon}.
\end{equation*}
If $m\in\{1,2\}$, we choose $\eta=4\varepsilon$, $\sigma_0=1+2\varepsilon$,
$\lambda=1/2-5\varepsilon$, $T=X^{1/2}$. This allows us to bound the sum of
\eqref{eq:error_1}--\eqref{eq:error_4} by
\begin{equation*}
  \ll_{\varepsilon}\absnorm(\ppp_1\cdots\ppp_\npid)^{3/16+\varepsilon}X^{1-13/32+6\varepsilon}
\end{equation*}
in case $m=1$ and, if $m=2$, by
\begin{equation*}
  \ll_{\varepsilon}\absnorm(\ppp_1\cdots\ppp_\npid)^{103/512+\varepsilon}X^{1-153/512+6\varepsilon}.
\end{equation*}
This concludes the proof of Theorem \ref{thm:count_quad_general} when $n=2$.

\subsection{The case $n\geq 3$}
In case
$n\geq 3$, a direct application of the truncated Perron formula would not yield
satisfactory bounds. Thus, as in standard proofs of tauberian theorems
(see, e.g., \cite[Appendice A]{MR1875171}), we consider first the weighted counting function
\begin{equation*}
  \tilde{N}_{\totram}(\Lambda,\PPP;X) := \sum_{m\leq X}a_m\log(X/m).
\end{equation*}
Moreover, we observe that
\begin{equation*}
  2b = \min\{1/2, 1/(2\phi(n)m\alpha)\},
\end{equation*}
with $\alpha = 103/256$ as in Propositon \ref{prop:ds_merom}. In the following
derivations, $\varepsilon$ denotes a small positive constant. Its precise value
may change between its occurrences. From 
(\ref{eq:ds_estimate}) with $\eta=\varepsilon$, we conclude that
 \begin{equation}\label{eq:general_estimate}
    \frac{\abs{s-1/(n-1)}}{\abs{s}}D(\Lambda,\PPP;s) \ll_{\varepsilon}
    \absnorm(\ppp_1\cdots\ppp_\npid)^{1/(2m)+\varepsilon}(1+|\Im
    s|)^{1/2+\varepsilon}
  \end{equation}
for $(n-1)\Re s\geq 1-2b+\varepsilon$.
  \begin{lemma}\label{lem:ntilde}
    We have the asymptotic formula
    \begin{equation*}
      \tilde{N}_{\totram}(\Lambda,\PPP;X) =
      (n-1)\delta_{\ppp_1\cdots\ppp_\npid}c_{F,n}X^{1/(n-1)} + O_\varepsilon(\absnorm(\ppp_1\cdots\ppp_\npid)^{1/(2m)+\varepsilon}X^{(1-2b)/(n-1)+\varepsilon}).     
    \end{equation*}
  \end{lemma}

  \begin{proof}
    Let $\sigma_0=1/(n-1)+\varepsilon$. The integral representation
    \begin{equation*}
       \tilde{N}_{\totram}(\PPP;X) =
       \frac{1}{2\pi i}\int_{\sigma_0-i\infty}^{\sigma_0+i\infty}D(\Lambda,\PPP;s)\frac{X^s}{s^2}\mathrm
       d s
    \end{equation*}
   converges absolutely due to \eqref{eq:general_estimate}. Moving the line of
   integration to 
   \begin{equation*}
\Re s = \sigma_1 := \frac{1-2b+\varepsilon}{n-1},
\end{equation*}
 we pick up the
   pole at $s=1/(n-1)$ with residue
   \begin{equation*}
     \Res_{s=\frac{1}{n-1}}\left(D(\Lambda,\PPP;s)\frac{X^s}{s^2}\right)=(n-1)\delta_{\ppp_1\cdots\ppp_\npid}c_{F,n}X^{1/(n-1)}.
   \end{equation*}
   Thus, $\tilde{N}_{\totram}(\PPP;X)-(n-1)\delta_{\ppp_1\cdots\ppp_\npid}c_{F,n}X^{1/(n-1)}$ equals
   \begin{equation*}
     \frac{1}{2\pi i}\lim_{T\to\infty}\left(-\int_{\sigma_1-iT}^{\sigma_0-iT}+\int_{\sigma_1-iT}^{\sigma_1+iT}
     + \int_{\sigma_1+iT}^{\sigma_0+iT}\right)D(\Lambda,\PPP;s)\frac{X^s}{s^2}\mathrm d
   s.
   \end{equation*}
   The two horizontal integrals tend to $0$ for $T\to\infty$, due to
   \eqref{eq:general_estimate}. The vertical integral becomes
   \begin{equation*}
     \int_{\sigma_1-i\infty}^{\sigma_1+i\infty}D(\Lambda,\PPP;s)\frac{X^s}{s^2}\mathrm d
   s\ll_\varepsilon
   X^{\sigma_1}\absnorm(\ppp_1\cdots\ppp_\npid)^{1/(2m)+\varepsilon}\left(1+\int_{t=1}^\infty
   t^{-3/2+\varepsilon}\mathrm
    d t\right),
   \end{equation*}
   which is covered by the Lemma's error term.
  \end{proof}
  
  We now deduce an asymptotic formula for $N(X):=N_{\totram}(\Lambda,\PPP;X)$ from the
  formula for $\tilde{N}(X):=\tilde{N}_{\totram}(\Lambda,\PPP;X)$. Clearly, for
  $0<u<1$, we have
  \begin{equation*}
    \frac{\tilde{N}(X(1-u))-\tilde{N}(X)}{\log(1-u)}\leq
    N(X)\leq \frac{\tilde{N}(X(1+u))-\tilde{N}(X)}{\log(1+u)}.
  \end{equation*}
  Moreover, for $u\in (-1,1)\smallsetminus\{0\}$, it follows from Lemma \ref{lem:ntilde} and the
  elementary inequality
  \begin{equation*}
    |u| \leq 2|\log(1+u)|\quad\text{ for }\quad|u|< 1,
  \end{equation*}
 that
  \begin{align*}
    \frac{\tilde{N}(X(1+u))-\tilde{N}(X)}{\log(1+u)} =
    &(n-1)\delta_{\ppp_1\cdots\ppp_\npid}c_{F,n}X^{1/(n-1)}\frac{(1+u)^{1/(n-1)}-1}{\log(1+u)}\\
    &+ O_\varepsilon(\absnorm(\ppp_1\cdots\ppp_\npid)^{1/(2m)+\varepsilon}X^{(1-2b)/(n-1)+\varepsilon}|u|^{-1}).
  \end{align*}
  With the further elementary estimate
  \begin{equation*}
    \frac{(1+u)^{1/(n-1)}-1}{\log(1+u)} = \frac{1}{n-1} +
    O_\varepsilon(u),\quad\text{ for }u\in[-1+\varepsilon,1]\smallsetminus\{0\},
  \end{equation*}
  and using (\ref{eq:lead_const}), Theorem \ref{thm:count_quad} follows immediately from the choice $u:=\pm X^{-b/(n-1)}$.

\section{Heights and small splitting primes}\label{sectionkeylemma}
Let 
\begin{alignat*}1
H_K(\alpha)=\prod_{v\in M_K}\max\{1,|\alpha|_v\}^{d_v}
\end{alignat*}
be the relative multiplicative Weil height of $\alpha\in K$.
Here $M_K$ denotes the set of places of $K$, and for each place $v$ we choose the unique representative $| \cdot |_v$ 
that either extends the usual archimedean absolute value on $\IQ$ or a usual $p$-adic absolute value on $\IQ$, and $d_v = [L_v : \IQ_v]$ denotes the local degree at $v$.
Note that this is exactly the height  in \cite[(2.2)]{EllVentorclass} for the principal divisor $(\alpha, (\alpha))$ associated to $\alpha\in K^\times$.
For an extension of number fields $K/F$ we also use the following invariant
\begin{alignat*}1
\del=\inf\{H_K(\alpha);K=F(\alpha)\},
\end{alignat*}
introduced in  \cite{8,Art1}, and also studied\footnote{In the cited works the authors used the absolute instead of the 
relative height, and denoted the invariant by $\delta(K/k)$ and $\delta(K)$ respectively.} in \cite{VaalerWidmer2013,VaalerWidmer}.
First we recall the key lemma \cite[Lemma 2.3]{EllVentorclass} of Ellenberg and Venkatesh; in fact we state a slightly more precise 
version. Recall from \cite{EllVentorclass} that a prime ideal $\mathfrak{B}$ of $\Oseen_K$ is said to be an extension of a prime ideal 
from a subfield $K_0\subsetneq K$ if there exists a prime ideal $\mathfrak{p}$ of $\Oseen_{K_0}$ such that $\mathfrak{B}=\mathfrak{p}\Oseen_K$. If $\mathfrak{B}$ and $\p$ are non-zero prime ideals in $\Oseen_K$ and $\Oseen_F$ respectively and $\mathfrak{B}\mid \p\Oseen_K$ then we say $\mathfrak{B}$ is unramified in $K/F$ if $\mathfrak{B}^2\nmid \p\Oseen_K$. 

\begin{proposition}[Ellenberg and Venkatesh]\label{keylemma}
Suppose $F\subseteq K$  are number fields, $[K:\IQ]=d$, $\del>\absnorm(\Delta(K/F))^\gamma$, $\delta<\gamma/\ell$, and $\varepsilon>0$.
Moreover, suppose $\mathfrak{B}_1,\ldots,\mathfrak{B}_M$ are unramified prime ideals in $K/F$ of norm
$\absnorm(\mathfrak{B}_i)\leq \absnorm(\Delta(K/F))^\delta$ and are not extensions of prime ideals from any proper subfield of $K$ containing $F$. 
Then we have 
$$\#Cl_K[\ell] \ll_{d,\ell,\gamma,\varepsilon} \DK^{1/2+\varepsilon}M^{-1}.$$
\end{proposition}
\begin{proof}
Exactly as in \cite[Lemma 2.3]{EllVentorclass} except that we replace their Lemma 2.2 by the hypothesis $\del>\absnorm(\Delta(K/F))^\gamma$.
\end{proof}

Recall that  $F$ is a number field of degree $m$ with algebraic closure $\Fbar$, and let $n\geq 2$ be an integer. We set
\begin{alignat}1\label{coll}
\SSS_{F,n}:=\{K\subseteq \Fbar; F\subseteq K\text{ and }[K:F]=n\}
\end{alignat}
for the collection of all field extensions of $F$ of degree $n$. For a subset $\SSS\subseteq \SSS_{F,n}$
we set 
\begin{equation*}
\B_\SSS(X;Y,M):=\left\{K\in \SSS;\ \DK\leq X,
\begin{aligned}
&\text{ at most $M$ prime ideals $\p$ in $\Oseen_F$ with}\\ &\text{ $\absnorm(\p)\leq Y$ split completely in $K$}
\end{aligned}
\right\}.
\end{equation*}
The following proposition is a slight adaption of \cite[Proposition 3.1]{ltor}.
The set up is chosen such that it applies in the most forward manner;
in particular, that is the reason why we introduce the  quantity $\tilde{\delta}_0$.

\begin{proposition}\label{mainprop}
Suppose $\SSS\subseteq \SSS_{F,n}$, and suppose $\gamma>0$ and $\theta\geq 0$ are such that
\begin{alignat*}1
\#\{K\in \SSS; \DK\leq X, \del\leq \DK^\gamma\}\ll_{\SSS,\gamma,\theta} X^\theta.
\end{alignat*}
Let $\varepsilon>0$, $\tilde{\delta}_0>0$, $\delta_0:=\min\{\gamma/\ell-2\varepsilon,\tilde{\delta}_0\}$, and $E_{\delta_0,\varepsilon}(\cdot)$ be an increasing function such that 
\begin{alignat*}1
\card\B_\SSS(X;X^{\delta_0},X^{\delta_0-\varepsilon})\leq E_{\delta_0,\varepsilon}(X).
\end{alignat*}

Then we have 
$$\#Cl_K[\ell] \ll_{[K:\IQ],\ell,\gamma,\varepsilon} \DK^{1/2-\delta_0+2\varepsilon}$$
for all but $O_{\SSS,\gamma,\theta,\varepsilon}((\log X)E_{\delta_0,\varepsilon}(X)+X^{\theta})$ fields $K$ in $\SSS$ with $\DK\leq X$.
\end{proposition}
\begin{proof}
The proof is essentially identical to the one of \cite[Proposition 3.1]{ltor}; in short,
use Proposition \ref{keylemma} and dyadic splitting. 
\end{proof}

\section{Ellenberg, Pierce and Wood's Chebyshev sieve}
In this section we describe the Chebysev sieve recently introduced by
Ellenberg, Pierce and Wood.  This is one of the key ideas in their
work \cite{EllenbergPierceWood} and allows them to show that almost
all number fields of degree $d$ have sufficiently many small splitting
primes, at least if $d\leq 5$ (and excluding $D_4$-fields if $d=4$).
Our sieve setting is slightly more general than the original one in
\cite{EllenbergPierceWood}, but no new arguments are needed.

Let $P(z):=\prod_{\absnorm(\p)\leq z}\p$ be the product of all (non-zero) prime ideals in $\Oseen_F$ of norm below $z$,
and let $\A$ be a finite set of cardinality $N$. To each prime $\p$ we associate a
property $(\p)$ that an element of $\A$ might have or not have.
We put 
\begin{equation*}
\A_\p:=\{a\in \A; a \text{ has property } (\p)\}, 
\end{equation*}
and for distinct primes $\p,\q$ we set $\A_{\p\q}:=\A_\p\cap\A_\q$.
Let $0\leq \delta_\p<1$ and $R_\p$ and $R_{\p,\q}$ such that
\begin{alignat*}1
\#\A_\p&=\delta_\p N+R_\p,\\
\#\A_{\p\q}&=\delta_\p\delta_\q N+R_{\p,\q}.
\end{alignat*}
Furthermore, we introduce 
\begin{alignat*}1
N(a):=\#\{\p|P(z);a\in\A_\p\}
\end{alignat*}
and its mean 
\begin{alignat}1\label{Mean}
M(z):=\frac{1}{N}\sum_{a\in \A}N(a)=\frac{1}{N}\sum_{\p|P(z)}\#\A_\p.
\end{alignat}
The quantity we want to bound from above is the number of $a\in \A$ for which 
$N(a)$ is significantly below the mean $M(z)$. For $M>0$ let us introduce 
\begin{alignat*}1
E(\A;z,M):=\#\{a\in \A; N(a)\leq M\}.
\end{alignat*}

In this setting their statement \cite[Proposition 3.1]{EllenbergPierceWood}  reads as follows.

\begin{lemma}[Ellenberg, Pierce, Wood]\label{sievelemma}
Suppose $M(z)>0$. Then we have 
\begin{alignat*}1
E(\A;z,\frac{1}{2}M(z))\leq\frac{4N}{M(z)^2}\left(U(z)+\frac{1}{N}\sum_{\p,\q|P(z)}|R_{\p,\q}| +\frac{2U(z)}{N}\sum_{\p|P(z)}|R_{\p}|+\left(\frac{1}{N}\sum_{\p|P(z)}|R_{\p}|\right)^2\right),
\end{alignat*}
where $U(z)=\sum_{\p|P(z)}\delta_\p$.
\end{lemma}
\begin{proof}
The proof is exactly the same as in  \cite{EllenbergPierceWood}.
\end{proof}

\subsection{Application of the Chebyshev sieve}\label{AppCheb}\ 

Let $\SSS\subseteq \SSS_{F,n}$, let $\E$ be a finite set of prime ideals
$\p$ in $\Oseen_F$, and set
\begin{equation*}
 \A:=\{K\in \SSS; \DK\leq X\}.
\end{equation*}
 For (non-zero)
prime ideals $\p$ in $\Oseen_F$ outside of $\E$ we let the property
$(\p)$ be ``$\p$ splits completely'', so that $\A_\p$ is the set of
fields $K$ in $\A$ in which $\p$ splits completely, and we set
$\A_\p=\emptyset$ if $\p\in \E$.  Let $\ef=\p$ or $\ef=\p\q$ for
distinct prime ideals $\p$ and $\q$ in $\Oseen_F$.  Put
$N_\SSS(X):=\#\A$, $N_\SSS(\ef;X):=\#\A_\p$ if $\ef=\p$ and
$N_\SSS(\ef;X):=\#(\A_\p\cap\A_\q)$ if $\ef=\p\q$.  Suppose that $c_\SSS >0$,
$0\leq \tau<\ro\leq 1$, and $\sigma\geq 0$, and that we have 
\begin{alignat}1\label{Scounting1}
N_\SSS(X)&=c_\SSS X^{\ro}+ O_{\SSS,\varepsilon}\left(X^{\tau+\varepsilon}\right), \\ 
\label{Scounting2}
N_\SSS(\ef;X)&=\delta_\ef c_\SSS X^{\ro}+ O_{\SSS,\varepsilon}\left((\absnorm(\ef))^{\sigma}X^{\tau+\varepsilon}\right), 
\end{alignat}
where $\delta_\ef$ is a multiplicative function with
$1\ll_{\SSS}\delta_\p\leq 1$ if $\p\notin \E$ and $\delta_\p=0$ if $\p\in \E$.

Note that with $N:=N_\SSS(X)$ we have 
\begin{alignat*}1
N_\SSS(\ef;X)&=\delta_\ef N+ O_{\SSS,\varepsilon}\left((\absnorm(\ef))^{\sigma}X^{\tau+\varepsilon}\right), 
\end{alignat*}
and hence,
\begin{alignat*}1
|R_\p|&=O_{\SSS,\varepsilon}\left((\absnorm(\p))^{\sigma}X^{\tau+\varepsilon}\right), \\
|R_{\p,\q}|&=O_{\SSS,\varepsilon}\left((\absnorm(\p\q))^{\sigma}X^{\tau+\varepsilon}\right).
\end{alignat*}

\begin{lemma}\label{sieveapplication}
Suppose that $\varepsilon>0$ and 
\begin{alignat*}1
\delta_0\leq\frac{\ro-\tau}{1+2\sigma}.
\end{alignat*}
Then we have
\begin{alignat*}1
\frac{X^{\delta_0}}{\log X}\ll_{\SSS,\delta_0,\E}M(X^{\delta_0})\ll_{\SSS,\delta_0,\E}\frac{X^{\delta_0}}{\log X},
\end{alignat*}
and 
\begin{alignat*}1
E(\A;X^{\delta_0},\frac{1}{2}M(X^{\delta_0}))\ll_{\SSS,\delta_0,\E,\varepsilon}X^{\ro-\delta_0+\varepsilon},
\end{alignat*}
provided $X$ is large enough in terms of $\SSS,\delta_0,\E,$ and $\varepsilon$.
\end{lemma}
\begin{proof}
We follow the proof of \cite[Proposition 6.1]{EllenbergPierceWood} with the obvious modifications. In this proof the implicit constants in the Vinogradov-symbols and in the $O(\cdot)$-notation depend only on $\SSS,\delta_0,\E,$ and $\varepsilon$. First we note that 
\begin{alignat*}1
\frac{1}{\#\A}\sum_{\p| P(z)}|R_\p|\ll z^{1+\sigma}X^{\tau-\ro+\varepsilon}, \quad \frac{1}{\#\A}\sum_{\p,\q| P(z)}|R_{\p,\q}|\ll z^{2+2\sigma}X^{\tau-\ro+\varepsilon}.
\end{alignat*}
Now 
\begin{alignat*}1
U(z)=\sum_{\p| P(z)}\delta_\p=\sum_{\p| P(z)\atop \p\notin \E}\delta_\p,
\end{alignat*}
and using Landau's prime ideal theorem we find that for $z\gg 1$ we have
\begin{alignat*}1
c_0z(\log z)^{-1}\leq U(z)\leq 2z(\log z)^{-1},
\end{alignat*}
for a constant $c_0>0$ depending on $\SSS$ and $\E$. Now we use (\ref{Mean}) to compute the mean
\begin{alignat*}1
M(z)=U(z)+\frac{1}{\#\A}\sum_{\p| P(z)}|R_\p|=U(z)+O_{\SSS,\varepsilon}(z^{1+\sigma}X^{\tau-\ro+\varepsilon}).
\end{alignat*}
From now on we assume that $z=X^{\delta_0}$. Noting that  $\delta_0\leq \frac{\ro-\tau}{1+2\sigma}<\frac{\ro-\tau}{\sigma}$ we conclude that for sufficiently large $X$ the last error term is bounded by $U(z)/2$, and hence  for 
$X\gg 1$ we get  
\begin{alignat*}1
c_1X^{\delta_0}(\log X)^{-1}\leq \frac{1}{2}U(X^{\delta_0})\leq M(X^{\delta_0})\leq \frac{3}{2}U(X^{\delta_0})\leq c_2X^{\delta_0}(\log X)^{-1}
\end{alignat*}
for constants $0<c_1<c_2\leq 1$ depending only on $\SSS$, $\delta_0$ and $\E$. Applying Lemma \ref{sievelemma} and simplifying terms yields 
\begin{alignat*}1
E(\A;X^{\delta_0},\frac{1}{2}M(X^{\delta_0}))\ll X^\varepsilon(X^{\ro-\delta_0}+X^{2\sigma\delta_0+\tau})\leq 2X^{\ro-\delta_0+\varepsilon}.
\end{alignat*}
\end{proof}
We conclude from Lemma \ref{sieveapplication}
that for $X\gg_{\SSS,\delta_0,\E,\varepsilon} 1$ we have $E(\A;X^{\delta_0},X^{\delta_0-\varepsilon})\leq E(\A;X^{\delta_0},\frac{1}{2}M(X^{\delta_0}))$,
and thus
\begin{alignat}1\label{Eestimate}
E(\A;X^{\delta_0},X^{\delta_0-\varepsilon})\ll_{\SSS,\delta_0,\E,\varepsilon} X^{\ro-\delta_0+\varepsilon}.
\end{alignat}

\section{Proofs of Theorems \ref{introthm:impGRH} and \ref{Thm3}}
In this section we fix $\SSS\subseteq \SSS_{F,n}$ and $\E$, and we suppose that we are in exactly the same setting as in \S\ref{AppCheb}. In particular, we have \eqref{Scounting1} and \eqref{Scounting2}, with a multiplicative function $\delta_\ef$ that satisfies $1\ll_{\SSS}\delta_\p\leq 1$ if $\p\notin \E$ and $\delta_\p=0$ if $\p\in \E$. Furthermore, we set 
\begin{alignat}1\label{defdeltatilde}
\tilde{\delta}_0:=\frac{\ro-\tau}{1+2\sigma}.
\end{alignat}

\begin{proposition}\label{mainprop1}
Let $\varepsilon>0$. Then we have 
\begin{alignat}1\label{mainprop1bound}
\#Cl_K[\ell] \ll_{\SSS,\ell,\varepsilon} \DK^{1/2-\min\{\frac{1}{2\ell(n-1)},\tilde{\delta}_0\}+\varepsilon}
\end{alignat}
for all but $O_{\SSS,\tilde{\delta}_0,\E,\varepsilon}(X^{\ro-\min\{\frac{1}{2\ell(n-1)},\tilde{\delta}_0\}+\varepsilon})$ fields $K$ in $\SSS$ with $\DK\leq X$.
\end{proposition}
\begin{proof}
We have $\card\B_\SSS(X;Y,M)\leq E(\A;Y,M)$, so that (\ref{Eestimate}) provides the required bound for  $\card \B_\SSS(X;X^{\delta_0},X^{\delta_0-\varepsilon})$.
Furthermore, from  \cite[Lemma 2.2]{EllVentorclass} (see also \cite{9} for an older and more
general result) we have $\del\gg_{F,n} \DK^{1/(2(n-1))}$.
Hence we can apply Proposition \ref{mainprop} with $\gamma=1/(2(n-1))+\varepsilon$, $\theta=0$, and $\tilde{\delta}_0$ as defined in (\ref{defdeltatilde}).
This completes the proof of Proposition \ref{mainprop1}.
\end{proof}

\subsection{Proof of Theorem \ref{Thm3}} Now we restrict ourselves to
those families for which Theorem \ref{thm:count_quad_general} provides
the required asymptotic formulas \eqref{Scounting1} and
\eqref{Scounting2}. For every place $v\mid n\infty$ of $F$, let $M_v$
be a set of Galois-extensions $K\subseteq \Fbar_v$ of $F_v$ with
cyclic Galois group of order dividing $n$, and assume that
$F_v\in M_v$. For $K\in \totram$, we write $K_v\in M_v$ if the
completion $K_w$ at any place $w$ of $K$ above $v$ is $F_v$-isomorphic
to a field in $M_v$. Writing $M=(M_v)_{v\mid n\infty}$, we consider
the family
\begin{equation*}
  \totram(M) := \{K\in\totram;\ K_v\in M_v \text{ for all $v\mid n\infty$}\}.
\end{equation*}
   We define 
  \begin{equation*}
    \tilde{\delta}'=\tilde{\delta}'(m,n):=\frac{b}{(n-1)(1+2a)},
  \end{equation*}
  where $a=a(m,n)$ and $b=b(m,n)$ are defined in Theorem
  \ref{thm:count_quad_general}. Using the bounds \eqref{eq:a_b_bounds}, one sees easily
  that $\tilde{\delta}' \geq \tilde{\delta}$, where $\tilde{\delta}$ is defined in (\ref{def:deltatilde}). Hence, the following theorem is a more precise and slightly more general version of Theorem \ref{Thm3}.

\begin{theorem}\label{thm:general}
Suppose $F$ and $\QQ(\mu_n(\Fbar))$ are linearly disjoint over $\QQ$, and $\varepsilon>0$. Then for all but $O_{F,n,\varepsilon}(X^{\frac{1}{n-1}-\min\{\frac{1}{2\ell(n-1)},\tilde{\delta}'\}+\varepsilon})$ fields $K$ in  $\totram(M)$ with $\DK\leq X$, we have 
\begin{alignat*}1
\#Cl_K[\ell] \ll_{F,n,\ell,\varepsilon} \DK^{\frac{1}{2}-\min\{\frac{1}{2\ell(n-1)},\tilde{\delta}'\}+\varepsilon}.
\end{alignat*}
\end{theorem}

\begin{proof}
  Take $\SSS:=\totram(M)$ and $\E$ to be the set of those prime ideals in
  $\Oseen_F$ that divide the ideal $n\Oseen_F$. Then by Theorem
  \ref{thm:count_quad_general} we have (\ref{Scounting1}) and (\ref{Scounting2}) with
  $\ro=1/(n-1)$, $\sigma=a+\varepsilon'$ and $\tau=(1-b)/(n-1)$, and
  $\delta_\ef$ is a multiplicative function with $\delta_\p$ as defined in
  (\ref{eq:lead_const}) if $\p\notin \E$ and $\delta_\p=0$ otherwise.
  Applying Proposition \ref{mainprop1} with 
  \begin{alignat*}1
\tilde{\delta}_0=\frac{\ro-\tau}{1+2\sigma}=\frac{b}{(n-1)(1+2a+2\varepsilon')}
\end{alignat*}

  and using that $\varepsilon'>0$ can be chosen arbitrarily small proves the theorem.
\end{proof}

\subsection{Proof of Theorem \ref{introthm:impGRH}: improving the GRH-bound}\label{SecImpGRH}

Here we lay out a general strategy to improve upon the GRH-bound
building on an idea from \cite{ltor}.  We still fix
$\SSS\subseteq \SSS_{F,n}$ and $\E$, and we continue to assume that we
have (\ref{Scounting1}) and (\ref{Scounting2}). The idea is to show
that for ``most'' extensions $K/F$ the lower bound for the crucial
quantity $\del$ is significantly bigger than Silverman's bound
$\DK^{1/(2(n-1))}$, and then to capitalise on this via Proposition
\ref{mainprop}.

We introduce the set of elements in $\Fbar$ that generate a field in $\SSS$
\begin{alignat*}1
P_\SSS :=\{\alpha\in \Fbar;F(\alpha)\in \SSS\},
\end{alignat*}
and its counting function
\begin{alignat*}1
N_H(P_\SSS, X):=\#\{\alpha\in P_\SSS ;H_{F(\alpha)}(\alpha)\leq X\}.
\end{alignat*}
If we have a good upper bound on this counting function then we can improve the
exponent in (\ref{mainprop1bound}).
\begin{proposition}\label{impGRH}
Suppose $\lambda$ is a real number such that
\begin{alignat*}1
N_H(P_\SSS, X)\ll_{\SSS,\lambda} X^\lambda,
\end{alignat*}
and suppose  $\gamma\geq 0$, and $\varepsilon>0$.
Then we have  
$$\#Cl_K[\ell] \ll_{\SSS,\ell,\gamma,\varepsilon} \DK^{1/2-\min\{\frac{\gamma}{\ell},\tilde{\delta}_0\}+\varepsilon}$$
for all but $O_{\SSS,\tilde{\delta}_0,\gamma,\lambda,\E,\varepsilon}(X^{\ro-\min\{\frac{\gamma}{\ell},\tilde{\delta}_0\}+\varepsilon}+X^{\lambda\gamma})$ fields $K$ in $\SSS$ with $\DK\leq X$.
\end{proposition}
\begin{proof}
Observe that the image of the map $\alpha\rightarrow F(\alpha)$ with domain 
$$\{\alpha\in P_\SSS ;H_{F(\alpha)}(\alpha)\leq X^\gamma\}$$ 
covers the set
$$\{K\in \SSS;\DK\leq X, \del\leq \DK^\gamma\}.$$ 
Using the hypothesis we conclude that
\begin{alignat}1\label{countingargument}
\#\{K\in \SSS;\DK\leq X, \del\leq \DK^\gamma\}\leq N_H(P_\SSS, X^\gamma)\ll_{\SSS,\lambda}X^{\lambda\gamma}.
\end{alignat}
As in the proof of Proposition \ref{mainprop1} we apply 
Proposition \ref{mainprop} with $\tilde{\delta}_0$ as defined in (\ref{defdeltatilde}), but this time with
$\theta=\lambda\gamma$.
This completes the proof of Proposition \ref{impGRH}.
\end{proof}

Let us now consider the special case $\SSS=\SSS_{F,n}$. Note that for $F=\IQ$ the cardinality of the set of algebraic numbers of degree $n$ over $F$ with height at most $X$ is bounded
from above by $n$ times the number of (irreducible) degree $n$ polynomials in $\IZ[x]$ of Mahler measure  at most $X$.
The Mahler measure in turn is bounded from below by $2^{-n}$ times the maximum norm of the coefficient vector.
This shows that for $F=\IQ$
\begin{alignat*}1
N_H(P_\SSS, X)\ll_n X^{n+1}.
\end{alignat*}
For arbitrary ground fields $F$ a similar argument applies (see \cite[Theorem]{22}) and provides 
\begin{alignat}1\label{degnumbers}
N_H(P_\SSS, X)\ll_{n,m}X^{n+1}.
\end{alignat}
Of course, this bound also holds for any $\SSS\subseteq \SSS_{F,n}$.
Applying  Proposition \ref{impGRH} with this bound proves the following theorem. 

\begin{theorem}\label{thm:impGRH}
Suppose $F$ is a number field and $\SSS\subseteq \SSS_{F,n}$ such that  (\ref{IntroScounting1}) and (\ref{IntroScounting2}) do hold.
Let $\varepsilon>0$ and $0\leq \gamma<1/(n+1)$. Then for all but 
\begin{equation*}
O_{\SSS,\tilde{\delta}_0,\gamma,\E,\varepsilon}(X^{1-\min\{\frac{\gamma}{\ell},\tilde{\delta}_0\}+\varepsilon}+X^{\gamma(n+1)}) 
\end{equation*}
fields $K$ in  $\SSS$ with $\DK\leq X$, we have 
\begin{alignat*}1
\#Cl_K[\ell] \ll_{\SSS,\ell,\gamma,\varepsilon} \DK^{\frac{1}{2}-\min\{\frac{\gamma}{\ell},\tilde{\delta}_0\}+\varepsilon}.
\end{alignat*}
\end{theorem}

Assuming the hypotheses of Theorem \ref{thm:impGRH} and taking $\gamma$ close enough to $1/(n+1)$ shows that for $\ell>1/(\tilde{\delta}_0(n+1))$, the bound  
(\ref{impGRHbound}) holds true for $100\%$ of $K\in \SSS$, counted by discriminant, thus proving Theorem \ref{introthm:impGRH}.

To get an improvement in Theorem \ref{Thm3} we would need that with $\SSS=\totram$
\begin{alignat}1\label{cyclicnumber}
N_H(P_\SSS,X)\ll_{\SSS}X^{\lambda}
\end{alignat}
for some $\lambda<2$. 
However, by Schanuel's Theorem \cite{25} even the contribution from a single field $K\in \SSS$ is already $\gg_{K}X^2$,
so that (\ref{cyclicnumber}) with $\lambda<2$ cannot be true.

\subsection{Further remarks}\ 

Finally, for $e\mid n$ let $\SSS_{F,n}(e)\subseteq \SSS_{F,n}$  be the subfamily of fields $K$ that contain a field $F\subseteq L\subseteq K$ of degree $[L:F]=e$.
For $\SSS=\SSS_{\IQ,n}(e)$ it follows immediately from \cite[Theorem 1.1]{art4} that
\begin{alignat}1\label{subfieldnumbers}
N_H(P_\SSS, X)\ll_{e,n}X^{e+n/e}.
\end{alignat}

Applying Proposition \ref{impGRH} with $\SSS=\SSS_{\IQ,n}(e)$, and using the above bound (\ref{subfieldnumbers})
proves that if  (\ref{IntroScounting1}) and (\ref{IntroScounting2}) do hold for $\SSS=\SSS_{\IQ,n}(e)$, then we have 

\begin{alignat*}1
\#Cl_K[\ell] \ll_{e,n,\ell,\varepsilon} \DK^{\frac{1}{2}-\min\{\frac{1}{\ell(e+n/e)},\tilde{\delta}_0\}+\varepsilon}
\end{alignat*}
for $100\%$ of all fields $K$ in  $\SSS_{\IQ,n}(e)$.
Particularly interesting is the case $n=e^2$. In this case we would get for all sufficiently large $\ell$  and $100\%$ of the fields $K\in \SSS_{F,n}(\sqrt{n})$ an improvement over the trivial
exponent by $1/(2\ell\sqrt{n})$ whereas in all other known cases of families the improvement decays like $O(1/n)$ as $n$ gets large.

Unfortunately, the required multiplicativity of
$\delta_{\mathfrak{e}}$ may not hold for the family $\SSS_{\IQ,n}(e)$,
no matter how we choose $\E$. However, note that if $\SSS$ is a
subfamily of $\SSS_{\IQ,n}(e)$ with linear growth rate, and we can
guarantee the existence of sufficiently many small splitting primes
for $100\%$ of all $K\in \SSS$ then Proposition \ref{keylemma}
combined with (\ref{countingargument}) and (\ref{subfieldnumbers})
provide an improvement to the GRH-bound for $100\%$ of all
$K\in \SSS$. For example, if we assume GRH then we have an improvement
over the trivial exponent by $1/(\ell(e+n/e))$ instead of just
$1/(2\ell(n-1))$.


\bibliographystyle{amsplain}
\bibliography{literature}

\end{document}